\documentclass[final]{article}

\usepackage[affil-it]{authblk}

\usepackage{amsmath,graphicx,latexsym,amssymb, amscd, psfrag,verbatim, amsfonts,amsthm}
\usepackage{amsfonts,bm,mathrsfs,subfigure}
\usepackage{algorithmic,algorithm,amsxtra,color}
\usepackage[algo2e,ruled,vlined]{algorithm2e}
\usepackage{geometry,graphicx}

\usepackage{setspace}

\newtheorem{theorem}{Theorem}
\newtheorem{lemma}{Lemma}

\newtheorem{remark}{Remark}

\newcommand{\bn}{\bm n}
\newcommand{\bw}{\bm w}

\newcommand{\bz}{\bm z}
\newcommand{\bx}{\bm x}

\newcommand{\bu}{\bm u}

\newcommand{\bv}{\bm v}
\newcommand{\bE}{\mathcal{E}}

\newcommand{\bg}{\bm g}

\newcommand{\bS}{\bm \sigma}

\newcommand{\bSS}{\bm S}
\newcommand{\bP}{\bm P}

\newcommand{\bF}{\bm F}
\newcommand{\bI}{\bm I}
\newcommand{\bJ}{\bm J}

\newcommand{\bphi}{\bm \phi}

\newcommand{\cA}{\mathcal A}

\begin{document}


\title{Modeling and Simulation for Fluid--Rotating Structure Interaction}

\author{Kai Yang\thanks{Department of Mechanical Engineering, Stanford University, 496 Lomita Mall, Stanford, CA 94305, USA.  Email:\texttt{kaiyang15@stanford.edu}},
Pengtao Sun\thanks{Corresponding author.  Department of Mathematical Sciences,
University of Nevada Las Vegas, 4505 Maryland Parkway, Las Vegas, NV
89154, USA.  Email:\texttt{pengtao.sun@unlv.edu}},  
Lu Wang\thanks{Lawrence Livermore National Laboratory,
7000 East Avenue, Livermore, CA 94550, USA. Email:\texttt{wang\_lu85@llnl.gov}},
Jinchao Xu\thanks{Principal corresponding author. Department of Mathematics, Pennsylvania State
University, University park, PA 16802, USA.  Email:\texttt{xu@math.psu.edu}},
Lixiang Zhang\thanks{Department of Mechanical Engineering, Kunming University of Science and technology,
Kunming, 68 Wenchang Road, Yunnan, China.  Email:\texttt{zlxzcc@126.com}}}
\maketitle
%
%
%
%
%
%


\begin{abstract}In this paper, we study a dynamic fluid-structure
interaction (FSI) model for an elastic structure that is immersed and
spinning in the fluid. We develop a linear constitutive model to
describe the motion of a rotational elastic structure which is
suitable for the application of arbitrary Lagrangian-Eulerian (ALE)
method in FSI simulation. Additionally, a novel ALE mapping method
is designed to generate the moving fluid mesh while the deformable
structure spins in a non-axisymmetric fluid channel.
The structure velocity is adopted as the principle unknown to form a
monolithic saddle-point system together with fluid velocity and
pressure. We discretize the nonlinear saddle-point system with mixed
finite element method and Newton's linearization, and prove that the
derived saddle-point problem is well-posed. The developed
methodology is applied to a self-defined elastic structure and a
realistic hydro-turbine under a prescribed angular velocity. Both
illustrate the satisfactory numerical results of an elastic structure
that is deforming and rotating while interacting with the fluid. The
numerical validation is also conducted to demonstrate the modeling
consistency.
\end{abstract}

{\bf Keywords:}
Fluid-rotating structure interaction,  arbitrary Lagrangian
Eulerian (ALE) method,  monolithic algorithm,  mixed finite
element method,  linear elasticity, master-slave relations.


\section{Introduction}
Fluid-structure interaction (FSI) problem remains as the one of the
most challenging problems in the computational mechanics and the
computational fluid dynamics. Researchers have conducted various
studies on certain types of FSI problems (such as fluid-rigid body
interaction
\cite{Sarrate.J;Huerta.A;Donea.J2001a,Desjardins.B2000a,Capdevielle.S2012a,Hu.H1996a,Johnson.A;Tezduyar.T1997a},
fluid with non-rotational structure
\cite{Hirth.C;Amsden.A;Cook.J1974a,Taylor.C;Hughes.T;Zarins.C1998a,Belytschko.T;Kennedy.J1978a,Belytschko.T;Kennedy.J;Schoeberle.D1980a,Nobile.F2001a,Xu.J;Yang.K2015a},
FSI with stationary fluid domain
\cite{Du.Q;Gunzburger.M;Hou.L;Lee.J2003a,Zhang.L;Guo.Y;Wang.W2009a,Zhang.L;Guo.Y;Zhang.H2010a,Zhang.L;Guo.Y;Wang.W2007a}).
However, there is a dearth of practical models on the FSI problems
involving a structure that rotates and deforms, i.e., an elastic
rotor. The development of mathematical model and numerical
methodology is critical in practice for large-scale advanced FSI
simulation involving an elastic rotor, as it largely benefits and
guides the design, evaluation and prediction of various
applications, such as the hydro-turbines, jet engine, and the
artificial heart pump. Therefore, it is of great significance for us
to develop an efficient and accurate mathematical model and
numerical method to handle the fluid-structure interaction involving
a rotational and deformable structure motion.

The difficulties associated with simulating the rotational structure
in FSI stem from the fact that the simulations rely on the coupling
of two distinct descriptions: the Lagrangian description for the
solid and the Eulerian coordinate for the fluid. The arbitrary
Lagrangian Eulerian (ALE) method \cite{Hirth.C;Amsden.A;Cook.J1974a,
Hughes.T;Liu.W;Zimmermann.T1981a, Huerta.A;Liu.W1988a,
Nitikitpaiboon.C;Bathe.K1993a, 2010a} copes with this difficulty by
adapting the fluid mesh to accomadate the deformations of the solid
on the interface. By using ALE method, the meshes of fluid and
structure are conforming on the interface if the Lagrangian
structure mesh is moved to the Eulerian one following ALE mapping.
This is important since the degrees of freedom on the interface are
naturally shared by both fluid and structure, which facilitates the
implementation of our discretization. However, ALE method has a
severe drawback, i.e., when the structure has a large displacement
or deformation, ALE mapping may very likely distort the fluid mesh.
Even the most advanced and best-tuned ALE-based scheme cannot
perform well without re-meshing. And, if the re-meshing is employed
to produce the fluid mesh, then the number of mesh nodes and/or
elements over different time levels can no longer be guaranteed to be
the same, and thus the interpolations of variables between every two
adjacent time steps are unavoidable, resulting in a time-consuming
and even unstable geometrical process, especially in the case of
high dimension. As a matter of fact, it is difficult to directly
apply ALE method to fluid-rotating structure
+interaction problems.

Many approaches have been proposed to deal with rotational structure
in FSI problems.  Several of these approaches model the wind turbine
rotor
\cite{Bazilevs.Y;Hsu.M;Takizawa.K;Tezduyar.T2012a,Bazilevs.Y;Hsu.M;Scott.M2012a,Hsu.M;Bazilevs.Y2012a,Bazilevs.Y;Hsu.M;Kiendl.J;Wuchner.R;Bletzinger.K2011a,Bazilevs.Y;Hsu.M2011a}
by coupling the finite element method (FEM) for fluid dynamics, the
isogeometric analysis (IGA) for structure mechanics, and the
non-conforming discretization on the interface of fluid and
structure. To apply ALE approach, they introduce an artificial
cylindrical buffer zone to enclose the rotor inside, and let the
mesh of the sliding cylindrical interface using weakly enforcement
of continuity of solution fields, and introduce extra unknowns
and/or penalties to reinforce the continuity on the interface by
means of, e.g., Lagrange multiplier or discontinuous Galerkin (DG)
method. Shear-slip method
\cite{Behr.M;Tezduyar.T1999a,Behr.M;Tezduyar.T2001a,Bazilevs.Y;Hsu.M;Takizawa.K;Tezduyar.T2012a}
was introduced to locally reconnect the mesh in order to keep the
mesh quality when the structure is undergoing translation or
rotation.



In this paper, we develop a new ALE method to produce a body-fitted
moving fluid mesh that is conforming with the rotational and
deformable structure mesh on the interface. And, to make our ALE
method work for the elastic rotor that is immersed in the fluid, we
first derive a linear structure equation that involves the
rotational matrix from the nonlinear structure model based upon a
decomposition of structure displacement into two components of
rotation and deformation. In addition to that, we define an
artificial cylindrical buffer zone in the fluid domain to embrace
the elastic rotor inside and rotate together on the same axis of
rotation with the same angular velocity. If the fluid channel is
non-axisymmetric, we then need to find out the relative motion
information between the rotational fluid subdomain (the cylindrical
buffer zone) and the stationary fluid subdomain (the rest part of
fluid domain) by matching the grid on the sliding interface, and
define them into our new ALE mapping. Finally, we develop a very
stable and easily attainable ALE method for the generation of a
rotational and deformable fluid mesh that matches with the structure
mesh on the interface.

Next, with the velocity instead of the displacement as the principle
unknown of structure, we define a monolithic saddle-point system for
the studied FSI problem, and further a monolithic algorithm to solve
the coupled fluid and structure equations. We prove the
well-posedness of the discrete linear saddle-point system resulting from mixed finite element discretization and Newton's
method. Numerical experiments are carried out for a self-defined
elastic rotor and a realistic hydro-turbine to illustrate that our
developed structure model and ALE-based monolithic method are
efficient and stable for the FSI problem involving an elastic rotor.
A numerical validation is also conducted to demonstrate the
consistency of our developed rotational structure model with respect
to the different structure parameters.

The paper is organized as follows. In Sections
\ref{sec:FSImodeling}-\ref{sec:ale}, we first define the general
governing equations, interface conditions and boundary conditions of
the FSI problem, and their weak formulations along with the ALE
techniques, then we introduce the monolithic weak formulation of
FSI. In Section \ref{sec:FSIwithRotation}, we develop an approximate
formulation of the constitutive equation for the rotating structure. Then, we
define a new ALE mapping for an elastic rotor that is immersed in
the fluid in Section \ref{sec:newALE}, and a monolithic numerical
discretization in Section \ref{sec:ALEdiscrete}, where, we also
analyze the well-posedness of the resulting discrete saddle-point
system. In Section \ref{section:algorithm}, we describe our
monolithic algorithm in detail. Then we show some numerical
experiments and conduct the numerical validations in Section
\ref{sec:numerics}. We draw the conclusions and outline the future
works in Section \ref{sec:conclusion}.

\section{Modeling of the fluid-structure interaction (FSI) problem}
\label{sec:FSImodeling}
\begin{figure}[htbp]
\begin{center}
\graphicspath{{:figure:}}
\includegraphics[width=0.5\textwidth]{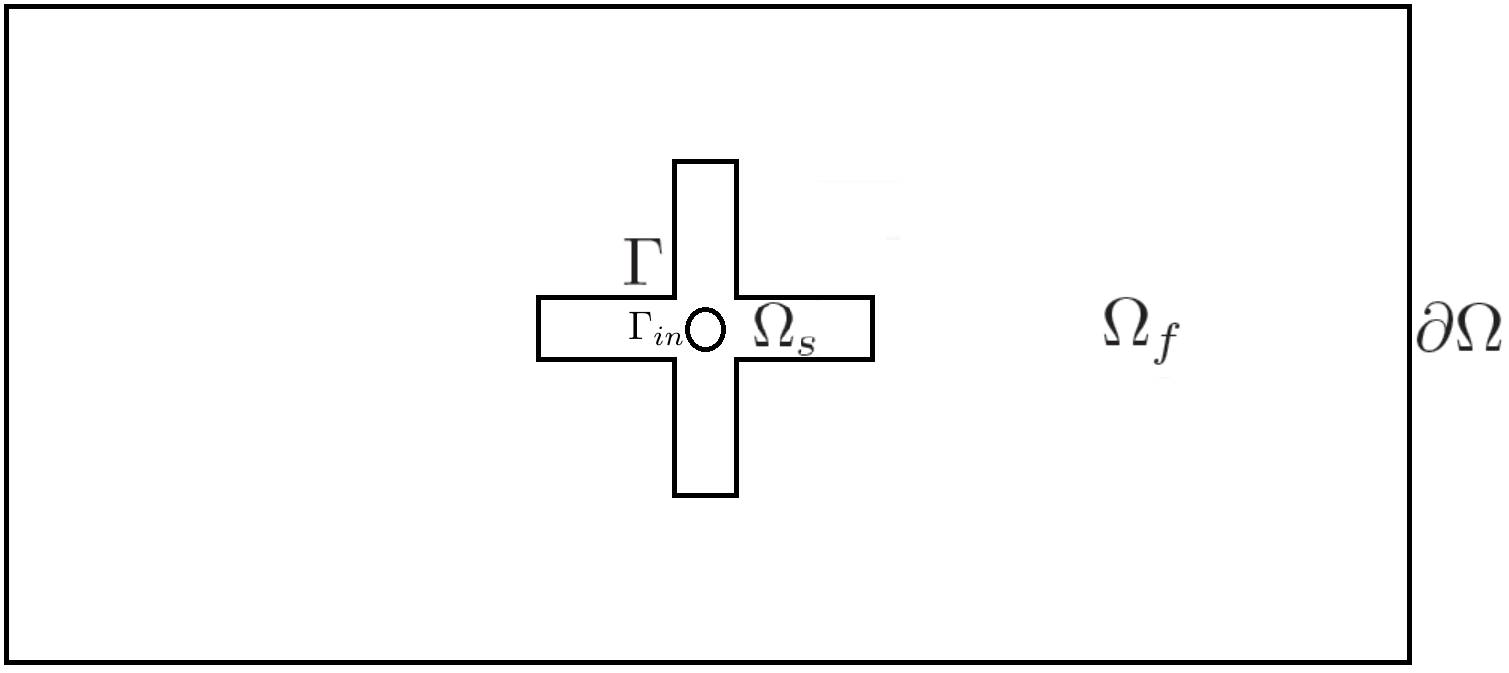}
\end{center}
\caption{A schematic domain of FSI involving an elastic rotor
$\Omega_s$ in 2D} \label{fig:rotationalFSI}
\end{figure}
%
%
%
%
%
%

As shown in Fig. \ref{fig:rotationalFSI}, we consider an elastic
rotor $\Omega_s(t)$ which is immersed in the fluid domain
$\Omega_f(t)$, and spins around its rigid axis of rotation displayed
as an inner circle $\Gamma_{in}$.
$\Omega_f(t)\cap\Omega_s(t)=\emptyset$. Denote the whole domain as
$\bar\Omega(t) = \bar\Omega_f(t)\cup \bar\Omega_s(t)\subset
\mathbb{R}^d\ (d= 2,3)$, and the interface of fluid and structure as
$\Gamma(t)=\partial\Omega_s(t)\cap\partial\Omega_f(t)$.  We use
$\partial_t$ to denote the time derivative $\frac{\partial}{\partial
t}$ in the rest of this paper.

For any position $\hat\bx\in\hat\Omega = \Omega(0)$, we denote by
$\bx(\hat\bx,t)\in\Omega(t)$ its position at time $t$.
$\bx(\hat\bx,t)$ is usually called flow map. We can also use
$\bx(\hat\bx,t)$ to denote the \emph{flow map}: $\hat\bx\mapsto
\bx(\hat\bx,t)$.  Then for given $t>0$, $\bx(\cdot,t)$ is a
diffeomorphism from $\Omega(0)$ to $\Omega(t)$.
%
For $({\hat{\bx}},t) \in \hat\Omega\times[0,T]$, we introduce the
following variables in \emph{Lagrangian coordinates} : the
displacement $\hat
{\bu}({{\hat{\bx}}},t)={\bx}({\hat{\bx}},t)-{\hat{\bx}}$, and the
velocity $\displaystyle  \hat {\bv}({\hat{\bx}},t)=\partial_t\hat
  {\bu}(\hat\bx,t)$.  Using the
relationship $\bx=\bx(\hat \bx,t)$, we also introduce the velocity
in \emph{Eulerian coordinates}: ${\bv}({\bx},t)=\hat
{\bv}({\hat{\bx}},t)$.  The symmetric part of the gradient of
velocity is denoted by $\displaystyle \bm\epsilon(\bv)=(\nabla
\bv+(\nabla \bv)^T)/2. $

For the sake of brevity of the notation, in the rest of this paper we
omit the temporal variable $t$ in those notations which are defined
in Eulerian coordinates, i.e., $\phi=\phi(t)$, and use $\hat\phi$ to
denote those variables which are defined in Lagrangian coordinates,
i.e., $\hat\phi=\phi(0)$ .


\paragraph{Fluid motion} The fluid is assumed to be Newtonian and incompressible.
It is described by the following conservations of momentum and mass,
i.e., Navier-Stokes equations, in terms of the fluid velocity
$\bv_f$ and pressure $p_f$ in Eulerian description.
\begin{equation}
\label{eq:fluid}
\left\{
\begin{aligned}
\rho_f(\partial_t\bv_f+\bv_f\cdot\nabla\bv_f)&=\nabla\cdot\bS_f+\bg_f, &\quad \text{in } \Omega_f,\\
\nabla\cdot\bv_f&=0, &\quad \text{in } \Omega_f,\\
\end{aligned}
\right.
\end{equation}
where
$$\bS_f=-p_f\bI+2\mu_f\bm\epsilon(\bv_f),$$
$\rho_f$ denotes the density of the fluid, $\mu_f$ the dynamic
viscosity of the fluid, and $\bg_f$ the external body force of the
fluid.
\paragraph{Structure motion}
The structure is assumed to be compressible. It is described by the
following conservations of momentum and mass, i.e., the dynamic
structure equations in terms of the structure velocity $\bv_s$ in
Eulerian description.
\begin{equation}
\label{eq:structure}
\left\{
\begin{aligned}
\rho_s(\partial_t\bv_s+\bv_s\cdot\nabla\bv_s)&=\nabla\cdot\bS_s+\bg_s, &\quad \text{in } \Omega_s,\\
\partial_t\rho_s+\nabla\cdot(\rho_s\bv_s)&=0, &\quad\text{in } \Omega_s,\\
\end{aligned}
\right.
\end{equation}
where $\rho_s$ denotes the density of the structure, and $\bg_s$ the
external body force of the structure. We define the Cauchy stress
tensor of the structure $\bS_s$ as follows by changing variables for
the structure equations from Eulerian coordinate to Lagrangian
coordinate. We then have the following structure equation defined in
Lagrangian system in terms of the structure displacement $\hat\bu_s$
in the initial (reference) domain $\hat\Omega_{s}$.
\begin{equation}
\label{eq:structure1}
\begin{aligned}
\hat\rho_s\partial_{tt}\hat\bu_s&=\nabla\cdot(\bJ\bS_{s}\bF^{-T})+\hat{\bg}_s,&\quad\text{in }\hat\Omega_s\\
\end{aligned}
\end{equation}
where  $\bF({\hat{\bx}},t)=\partial
  {\bx}/\partial {\hat{\bx}}=\bI+\nabla_{{\hat\bx}} \hat\bu_s$ is the
deformation gradient tensor, $\displaystyle
\bJ({\hat{\bx}},t)=det({\bF}({\hat{\bx}},t))$,
$\hat{\rho}_s=\bJ\rho_s$, and $\hat{\bg}_s=\bJ\bg_s$. The derivation
of (\ref{eq:structure1}) will be further shown in Section
\ref{sec:weakformFSI}.

We introduce the first Piola-Kirchhoff stress tensor $\bP_s$,
defined as $\bP_s=\bJ\bS_{s}\bF^{-T}$, to simplify the stress form
in (\ref{eq:structure1}). Further, $\bP_s=\bF\bSS$, where $\bSS$ is
called the second Piola-Kirchhoff stress tensor, and thus defined as
$\bSS=\bJ\bF^{-1}\bS_{s}\bF^{-T}$. Hence, the Cauchy stress tensor
$\bS_s$ can be defined as
$$
\bS_{s}(x)=\bJ^{-1}\bF\bSS\bF^T,
$$
which can be further specified based upon the constitutive law of a
given structure material, i.e., its the second Piola-Kirchhoff
stress tensor $\bSS$.

If we select the nonlinear St. Venant-Kirchhoff (STVK) material as
the material of elastic rotor that is popularly adopted in the
problems of hydro-turbine, wind-turbine and etc, we have the
constitutive law of STVK material as follows
\cite{Sugiyama.K;Ii.S;Takeuchi.S;Takagi.S;Matsumoto.Y2011a,Richter.T2013a}
$$\bSS=2\mu_s\bE+\lambda_s (tr\bE)\bI,$$
where, $\bE$ is the Green-Lagrangian finite strain tensor, defined
as
$$ \bE=(\bF^T\bF-\bI)/2=\frac{1}{2}\left(\nabla_{{\hat\bx}} \hat\bu_s+(\nabla_{{\hat\bx}}
\hat\bu_s)^T+(\nabla_{{\hat\bx}} \hat\bu_s)^T\nabla_{{\hat\bx}}
\hat\bu_s\right),
$$
$\lambda_s=\frac{E\nu}{(1+\nu)(1-2\nu)}$ is the Lam\'{e}'s first
parameter, and $\mu_s=\frac{E}{2(1+\nu)}$ is the shear modulus, here
$E$ and $\nu$ denote the Young's modulus and Poisson's ratio,
respectively. Hence, we can particularly define
$$
\bS_{s}(x)=\frac{1}{\bJ}\bF(2\mu_s \bE+\lambda_s (tr\bE)\bI)\bF^T.
$$
Then, the structure equation (\ref{eq:structure1}) can be rewritten
as the following form for STVK material, in particular.
\begin{equation}\label{solid-eq-new0}
\hat\rho_s \partial_{tt} \hat{\bm u}_s=\nabla\cdot (2\mu_s{\bm
F}{\bE}+\lambda_s(tr{\bE}){\bm F})+\hat\bg_s.
\end{equation}
Apparently, (\ref{solid-eq-new0}) is a nonlinear equation. In
Section \ref{sec:FSIwithRotation}, we will develop a linearized
model, which is an approximation to (\ref{solid-eq-new0}) under the
assumption of small deformation, specifically for a deformable
structure undergoing a large rotation.

\paragraph{Interface conditions}
Suppose the no-slip type boundary conditions hold on the interface
of the fluid and structure, i.e., both velocity and normal stress
are continuous across the interface $\Gamma$. The interface
conditions on $\Gamma$ can be defined as follows in Eulerian
description.
\begin{eqnarray}
\bv_f=\bv_s=\partial_t\bu_s, &\text{on } \Gamma, \label{interface}\\
\bS_f\bn_f+\bS_s\bn_s=0, &\text{on } \Gamma, \label{interface1}
\end{eqnarray}
where $\bn_f$ and $\bn_s$ are the unit normal vector to the
interface $\Gamma$ from the fluid and from the structure,
respectively.

\paragraph{Boundary conditions}
To fully define the boundary conditions other than the interface
conditions, we split the outer boundary of fluid to two types:
$\partial\Omega_f\backslash\Gamma=\Gamma_{D,f}\cup\Gamma_{N,f}$ and
$\Gamma_{D,f}\cap\Gamma_{N,f}=\emptyset$, where, $\Gamma_{D,f}$
represents a part of the outer boundary of fluid on which Dirichlet
boundary condition is imposed, and $\Gamma_{N,f}$ represents the
rest part on which Neumann boundary condition is defined. Similarly,
we split the outer boundary of structure as
$\partial\hat\Omega_s\backslash\hat\Gamma =
\hat\Gamma_{D,s}\cup\hat\Gamma_{N,s}$ and
$\hat\Gamma_{D,s}\cap\hat\Gamma_{N,s}=\emptyset$.

The Dirichlet and Neumann boundary conditions for fluid and
structure are given as follows
\begin{eqnarray}
\bv_f= \bv_{D,f}, &\text{on } \Gamma_{D,f},\label{eq:DirichletBC_f}\\
\hat \bu_s=\hat \bu_{D,s}, &\text{on }
\hat\Gamma_{D,s},\label{eq:DirichletBC_s}\\
\bS_f\bn_f= \bg_{N,f}, &\text{on } \Gamma_{N,f},\label{eq:NeumannBC_f}\\
\bS_s\bn_s=\hat \bg_{N,s}, &\text{on }
\hat\Gamma_{N,s}.\label{eq:NeumannBC_s}
\end{eqnarray}
where, $\bv_{D,f}(\bx,t),\ \hat\bu_{D,s}(\hat\bx,t),\
\bg_{N,f}(\bx,t),\ \hat\bg_{N,s}(\hat\bx,t)$ are all prescribed. For
example, Dirichlet boundary conditions for fluid can be the inflow
boundary condition at the inlet and no-slip boundary condition on
the wall. Dirichlet boundary conditions for structure can be used
when it is clamped to the boundary. With $\bg_{N,f}=0$, the Neumann
boundary condition is adopted for fluid as the stress free boundary
condition at the outlet of flow channel. Neumann boundary condition
is employed for structure in FSI when the structure embraces the
fluid inside and is imposed an external force on its outer boundary.

In our case studied in this paper, the structure is immersed in the
fluid, thus Neumann boundary condition is not applied to the
structure, i.e., $\hat\Gamma_{N,s}=\emptyset$. And,
$\hat\Gamma_{D,s}=\Gamma_{in}$, on which $\hat\bu_{D,s}$ is given if
a rotation is prescribed. Later in this paper, we use the structure
velocity, $\hat\bv_s$, as the principle unknown instead of
$\hat\bu_s$, noting that $\hat\bv_s=\partial_t\hat\bu_s$. Then the
corresponding Dirichlet boundary condition is changed as
\begin{equation}
\label{eq:DirichletBC_sv}
\hat \bv_s=\partial_t\hat \bu_{D,s}, \mbox{ on } \hat\Gamma_{D,s}.
\end{equation}
\paragraph{Initial conditions}
The following initial conditions hold for fluid and structure
equations which are defined on the time interval $[0,T]$, with
$T>0$.
\begin{equation}\label{eq:Initial}
\begin{array}{ll}
\bv_f(0)=\bv^0_f(\bx),&\text{in } \Omega_f\\
\hat\bu_s(0)=\hat\bu_s^0(\hat\bx),\
\hat\bv_s(0)=\hat\bv_s^0(\hat\bx),&\text{in } \hat\Omega_s,
\end{array}
\end{equation}
where $\hat\bu_s^0(\hat\bx)$, $\hat\bv_s^0(\hat\bx)$, and
$\bv_f^0(\bx)$ are all prescribed.

\section{Weak formulation of the FSI problem}
\label{sec:weakformFSI} Now we derive the weak formulations of the
FSI problem. First, for any given domain
$\Omega\subset\mathbb{R}^d$, we define the $H^1$ space of vector
functions as
$${\mathbf H}^1(\Omega)= (H^1(\Omega))^d.
$$
Moreover, define
\begin{align*}
{\mathbf H}^1_0(\Omega_f)&:=\{\bu\in {\mathbf H}^1(\Omega_f)| \bu=0, \text{ on } \partial \Omega\cap\partial\Omega_f\},\\
{\mathbf H}^1_0(\hat\Omega_s)&:=\{\bu\in {\mathbf
H}^1(\hat\Omega_s)| \bu=0,
\text{ on } \partial \hat\Omega\cap\partial\hat\Omega_s\},\\
{\mathbf H}^1_0(\hat\Omega_f)&:=\{\bu\in {\mathbf
H}^1(\hat\Omega_f)| \bu=0,
\text{ on } \partial \hat\Omega\cap\partial\hat\Omega_f\}.\\
\end{align*}
Let us define the following Sobolev spaces:
\begin{equation}
\begin{aligned}
  \label{spaceV}
\mathbb{V}&:=\{(\bv_f,\hat\bv_s)\in {\mathbf H}^1_0(\Omega_f)\times
{\mathbf H}^1_0(\hat\Omega_s) \text{ such that }
\bv_f\circ\bx(\hat\bx,t)=\hat\bv_s, \text{ on }\hat\Gamma\},\\
\mathbb{Q}&:=L^2(\Omega_f).
\end{aligned}
\end{equation}

In order to formulate the FSI problem weakly, we use test functions
defined in $\Omega$, With the test function $\bm \phi\in
(H_0^1(\Omega))^d$, we first write the weak formulations of the
fluid and structure equations, respectively, as follows.
$$
\int_{\Omega_f}\rho_fD_t \bv_f\bm \phi d\bx+\int_{\Omega_f}\bm
\sigma_f:\bm \epsilon(\bm \phi)d\bx-\int_{\Gamma}\bm
\sigma_f\bn_f\cdot\bm\phi d\bx =\int_{\Omega_f}\bg_f\bm\phi d\bx,
$$
$$\int_{\Omega_s}\rho_s D_t
\bv_s\bm\phi d\bx+\int_{\Omega_s}{\bm\sigma}_s:\bm\epsilon(\bm\phi)
d\bx-\int_{\Gamma}{\bm\sigma}_s\bn_s\cdot\bm\phi d\bx
=\int_{\Omega_s}\bg_s\bm\phi d\bx,
$$
where $D_t\bm\psi=\partial_t\bm\psi+\bm\psi\cdot\nabla\bm\psi$.

We add up these two equations and cancel the boundary integral terms
because of the interface condition (\ref{interface1}), resulting in
$$
\int_{\Omega_f}\rho_f D_t \bv_f\bm\phi
d\bx+\int_{\Omega_f}\bm\sigma_f:\bm\epsilon(\bm\phi)d\bx+\int_{\Omega_s}\rho_sD_t
\bv_s\bm\phi d\bx+\int_{\Omega_s}{\bm\sigma}_s:\bm\epsilon(\bm\phi)
d\bx=\int_{\Omega_f}\bg_f\bm\phi d\bx+\int_{\Omega_s}\bg_s\bm\phi
d\bx.
$$
By a change of coordinates $\bx=\bx(\hat{\bx},t),$ the stress term
of structure part can be rewritten in Lagrangian coordinates
$$
\int_{\Omega_s}{\bm\sigma}_s:\bm\epsilon(\bm\phi)d\bx=\int_{\hat{\Omega}_s}
\hat{\bm\sigma}_s:\nabla_{\hat{\bx}}\hat{\bm\phi}\bF^{-1}\bJ
d\hat{\bx}=\int_{\hat{\Omega}_s}(\bJ\hat{\bm\sigma}_s\bF^{-T}):\nabla_{\hat{\bx}}\hat{\bm\phi}
d\hat{\bx},
$$
where $\hat{\bm\phi}(\hat\bx,t)=\bm\phi(\bx(\hat\bx,t),t)$ and
$\hat{\bm\sigma}_s(\hat\bx,t)={\bm\sigma}_s(\bx(\hat\bx,t),t)$.   We
also change the coordinates for the inertial term and the body force
term, we then get the following weak form of FSI problem.
\begin{equation}
  \label{FSI-vari}
\int_{\Omega_f}\left(\rho_f D_t \bv_f\bm\phi
+\bm\sigma_f:\bm\epsilon(\bm\phi)\right) d\bx+\int_{\hat
  \Omega_s}\left(\hat\rho_s\partial_{tt} \hat\bu_s\hat{\bm\phi}+\bP_s:\nabla_{\hat\bx} \hat{\bm\phi} \right)d\hat\bx
  =\int_{\Omega_f}\bg_f\bm\phi d\bx+\int_{\hat \Omega_s}\hat \bg_s \hat{\bm\phi} d\hat\bx,
\end{equation}
which holds for any $\bm\phi\in \mathbb{V}$. Here, as shown in
(\ref{eq:structure1}) in Section \ref{sec:FSImodeling}, we also
derive the first Piola-Kirchhoff stress tensor
$\bP_s=\bJ\hat{\bm\sigma}_s\bF^{-T}$, the density of the structure
$\hat\rho_s(\hat\bx,t)=\bJ(\hat\bx,t) \rho_s(\bx(\hat\bx,t),t)$, and
the external body force $\hat\bg_s(\hat\bx,t)=\bJ(\hat\bx,t)
\bg_s(\bx(\hat\bx,t),t)$. By the conservation of mass, $\hat\rho_s$
is independent of $t$.

Thus, we are able to define the weak formulations of FSI as follows.
Find $\bv_f$, $p$ and $\hat\bu_s$, satisfying
(\ref{eq:DirichletBC_f}), (\ref{eq:DirichletBC_s}), and
(\ref{eq:Initial}) such that for any given $t>0$, the following
equations hold for any $(\bm\phi,\hat{\bm\phi})\in \mathbb{V}$ and $
q\in \mathbb{Q}$
\begin{equation}
\label{eq:FSI_weak} \left\{
\begin{aligned}
(\hat\rho_s\partial_{tt}\hat\bu_s,\hat{\bm\phi})_{\hat\Omega_s}+\left(\rho_fD_t\bv_f,\bm\phi\right)_{\Omega_f}+(\bP_s(\hat\bu_s),\nabla\hat{\bm\phi})_{\hat\Omega_s}&+(\bm\sigma_f,\bm\epsilon(\bm\phi))_{\Omega_f}\\
&=(\hat{\bg}_s,\hat{\bm\phi})+(\bg_f,\bm\phi),\\
(\nabla\cdot \bv_f,q)_{\Omega_f}&=0,\\
\bv_f\circ \bx(\hat\bx,t)&=\partial_t \hat\bu_s,\quad \text{on } \hat \Gamma.\\
\end{aligned}
\right.
\end{equation}
where, the pair $(\cdot,\cdot)_\Omega$ stands for the $L^2$-inner
product on a domain $\Omega$, and $\nabla{\bm\phi}=\nabla_{\bx}
\bm\phi,\ \nabla\hat{\bm\phi}=\nabla_{\hat\bx} \hat{\bm\phi}$.


The above weak formulation shows that, on the continuous level, the
continuity of velocity of fluid and structure across the interface
is enforced in the definition of Sobolev space $\mathbb{V}$.
However, on the discrete level, we only need one set of degrees of
freedom of velocity to be shared by fluid and structure on the
interface. To that end, first of all, a conforming mesh across the
interface is needed, which is guaranteed by ALE method. Secondly, we
need to merge the two sets of degrees of freedom of velocity arising
from both fluid and structure on the interface to be one set only,
which can be done by means of the so-called ``Master-Slave
Relation'' technique. Thus, we no longer have any redundant degrees
of freedom to eliminate, nor the need of Lagrange multiplier to
enforce the continuity of velocity. We will further discuss this
technique combining with ALE method in the next section.

\section{Application of ALE method to the FSI problem}
\label{sec:ale} Due to the interaction between fluid and structure,
we need to solve Navier-Stokes equations on a moving fluid domain.
Re-meshing is always an option for a moving domain but it introduces
extra computational cost and needs the interpolation to transfer the
data between different time steps. Another frequently used approach
is so-called the arbitrary Lagrangian-Eulerian (ALE) method, with
which the moving fluid mesh sticks to the moving structure mesh and
is updated by preserving the number of mesh nodes always the same,
thus no interpolation occurs. The essential idea of ALE method is to
formulate and solve the fluid problem on a deforming mesh, which
deforms with the structure at the interface and then the fluid mesh
is smoothed within the fluid domain. First introduced for finite
element discretizations of the incompressible fluids in
\cite{Hughes.T;Liu.W;Zimmermann.T1981a,Donea.J;Giuliani.S;Halleux.J1982a},
the ALE method provides an approach to find the fluid mesh that can
fit the moving fluid domain $\Omega_f$. This mapping is a
diffeomorphism on the continuous level, and we use piecewise
polynomials to approximate it on the discrete level.  We assume that
the mesh motion is piecewise linear in the rest of this paper in
order to make sure the updated mesh by ALE mapping remains a
undistorted triangular mesh.

We denote the image of $\hat\Omega_f$ under the piecewise linear map
$\bx_{h,f}$ by $\Omega_f^n$.  $\Omega_f^n$ is triangulated by a
moving mesh with respect to time, denoted by $T_h(\Omega_f^n)$. Note
that $\Omega_f^n$ is assumed to be a polygonal domain in 2D, and a polyhedral
domain in 3D.
There are two main ingredients in the ALE approach:
\begin{enumerate}
\item Define how the grid is moving with respect to time such that it matches
  the structure displacement at the fluid-structure interface.
\item Define how the material derivatives are discretized on the moving grid.
\end{enumerate}

Given the structure trajectory $\bx(\hat \bx,t)$ defined on
$\hat\Gamma,$ the moving grid can be described by a diffeomorphism
$\cA(\cdot,t):\hat\Omega_f\mapsto \Omega_f$ that satisfies
\begin{equation}
\label{eq:ALE} \left\{
\begin{aligned}
\cA(\hat \bx,t)&=\hat \bx, &\text{ on }& \partial\hat\Omega_f\cap\partial\hat\Omega,\\
\cA(\hat \bx,t)&=\bx(\hat \bx,t), &\text{ on }& \hat \Gamma.\\
\end{aligned}
\right.
\end{equation}
ALE mappings satisfying (\ref{eq:ALE}) are by no means unique. In
the interior of $\hat \Omega_f$, the ALE mapping can be
``arbitrary".  One popular approach to uniquely determine $\cA$ is
to solve a partial differential equation
\begin{equation}\label{ALEdefination}
\mathcal{L}\cA=0,\quad\mbox{in }\hat\Omega_f.
\end{equation}
A popular choice for the operator $\mathcal L$ is the Laplacian,
$\mathcal{L}=-\Delta$. For more choices of formulating the ALE
mapping equation, we refer to
\cite{Bazilevs.Y;Takizawa.K;Tezduyar.T2013a,Donea.J;Giuliani.S;Halleux.J1982a}
and references therein.

With the ALE technique, we can reformulate the fluid momentum
equation on a moving fluid domain as follows
\begin{equation}\label{NS_ALE}
\rho_f\partial^{\mathcal A}_t
\bm{v}_f+\rho_f((\bm{v}_f-\bm{w})\cdot\nabla)\bm{v}_f=\nabla\cdot\bS_f+\bg_f,
\end{equation}
where,
$\bm{w}=\partial_t\mathcal A\circ\mathcal A^{-1}$ denotes the
velocity of fluid mesh, and

$$\partial^{\mathcal A}_t \bm{v}_f|_{(\bx,t)}:=\partial_t \bm{v}_f+(\bm{w}\cdot\nabla) \bm{v}_f$$ is called ALE
material derivative.

In particular, we define the discretization of the ALE material
derivative on the moving grid as follows
\begin{equation}\label{ALE-MOC}
\partial^{\mathcal A}_t
\bm{v}_f\approx \frac{\bm{v}_f(\bx,t_n)-\bm{v}_f(\mathcal A_{n-1}\circ
A_{n}^{-1}(\bx),t_{n-1})}{\Delta t},
\end{equation}
where, $\mathcal A_{n-1}\circ \mathcal A_{n}^{-1}(\bx)$ is the location of $\bx$ at $t=t_{n-1}$ before it is moved by ALE mesh motion.  For example, if $\cA\circ\cA_n^{-1}(\bx)$ is a mesh node at $t=t_{n-1}$, the ALE mesh motion moves it to $\bx$ at $t=t_n$, which is a mesh node at $t=t_n$.
Thus the interpolation between the meshes on different time levels
is avoided.


In order to precisely fulfill the kinematic interface condition
(\ref{interface}) and apply the Master-Slave Relation, we need to
reformulate the structure equation in terms of the structure
velocity $\hat\bv_s$ instead of its displacement $\hat\bu_s$ in
(\ref{eq:FSI_weak}). We use the following relation of $\hat\bv_s$
and $\hat\bu_s$
\begin{equation}\label{velocity_displacement}
\hat{\bm v}_s=\partial_t\hat{\bm u}_s,\, \text{ or },\, \hat{\bm
u}_s=\hat{\bm u}_s^0+\int_0^t\hat{\bm v}_s(\tau)d\tau
\end{equation}
to reformulate (\ref{eq:FSI_weak}), and rewrite the kinematic
condition (\ref{interface}) in reference configuration as
$$
\hat\bv_s({\hat\bx})=\bv_f(\cA({\hat\bx},t)),\quad \text{on}\
\hat\Gamma.
$$
Thus the Master-Slave Relation can be efficiently applied since both
master and slave DOFs are simply equal to each other at the same
mesh node on the interface. In addition, we introduce an equivalent
ALE mapping for the displacement of fluid mesh $\cA_{\bu}(\hat
\bx,t)=\cA(\hat \bx,t)-\hat\bx(\hat \bx,t)$, satisfying
\begin{eqnarray}
-\Delta \cA_{\bu}&=0, &\text{ in }\hat \Omega_f\label{eq:ALE_displacement}\\
\cA_{\bu}&=0, &\text{ on } \partial\hat\Omega_f\cap\partial\hat\Omega,\label{eq:ALE_bc1}\\
\cA_{\bu}&=\hat\bu_s, &\text{ on } \hat \Gamma.\label{eq:ALE_bc2}
\end{eqnarray}
The resulted new weak formulation is thus defined as follows with
ALE mapping, the structure velocity $\hat\bv_s$ in the structure
equation, and ALE material derivative in the fluid equation. Find
$\bv_f,\hat\bv_s, p,\cA_{\bu}$, satisfying (\ref{interface}),
(\ref{eq:DirichletBC_f}), (\ref{eq:DirichletBC_sv}) and
(\ref{eq:Initial}), such that for any given $t>0$,  the following
system of equations hold for $\forall (\hat\bphi,\bphi)\in
\mathbb{V}, q\in\mathbb{Q} ,\xi\in {\mathbf H}^1_0(\hat\Omega_f)$
\begin{equation}\label{ELALE}
\left\{
\begin{aligned}
(\hat\rho_s\partial_t\hat{\bm v}_s,\hat \bphi)_{\hat\Omega_s}+({\bm
 P}_s(\hat{\bm u}_s^0+\int_0^t\hat{\bm
v}_s(\tau)d\tau),\nabla\hat\bphi)_{\hat\Omega_s}&+(\rho_f\partial^{\mathcal A}_t{\bm v}_f,\bphi)_{\Omega_f} \\
+(\rho_f({\bm
v}_f-\bw)\cdot \nabla {\bm v}_f,\bphi)
_{\Omega_f}+({\bm \sigma}_f,\bm\epsilon(\bphi))_{\Omega_f}&=( \hat{\bg}_s,\hat{\bm\bphi})+(\bg_f,\bm\bphi),\\
(\nabla\cdot {\bm v}_f,q)_{\Omega_f}&=0,\\
(\nabla\mathcal{A}_{\bu},\nabla \xi)_{\hat\Omega_f}&=0, \\
\cA_{\bu}&=\hat\bu_s^0+\int_0^t\hat\bv_s(\tau)d\tau,\mbox{ on }\hat\Gamma.\\
\end{aligned}
\right.
\end{equation}

For the simplicity of notation, in what follows, we suppose
$\bg_f=\hat\bg_s=0$ by assuming no external body force is acted on
the fluid and structure.


\section{Remodeling of the rotational elastic structure}
\label{sec:FSIwithRotation} In this section, we attempt to develop a
new model for an elastic rotor with large rotation and small
deformation, which, simultaneously, is suitable for the
implementation of ALE method for FSI problems which involve the
rotation as well as the deformation in the structure.

From (\ref{solid-eq-new0}) we know that the equation of an elastic
structure that follows the constitutive law of STVK material and is
spinning around the axis of rotation, $\Gamma_{in}$, is specifically
defined as
\begin{equation}\label{solid-eq-new}
\hat\rho_s \partial_{tt} \hat{\bm u}_s=\nabla\cdot\bP_s=\nabla\cdot
(\lambda_s(tr{\bE}){\bm F}+2\mu_s{\bm F}{\bE})
\end{equation}
with the following boundary condition
\begin{equation}\label{eq:displacement_bc}
\hat{\bm u}_s={\bm x}-{{\hat\bx}}=(R(\theta)-{\bm
I})({{\hat\bx}}-\hat{\bm x}_0), ~~\mbox{ on }\Gamma_{in},
\end{equation}
where $R(\theta)$ is the rotational matrix depending on the angle of
rotation $\theta(t)$. For instance, let us specify a case in which
the axis of rotation is $z$-axis, then
\begin{equation}\label{rotationmatrix}
R(\theta)=\left(
 \begin{array}{rrr}
 \cos\theta&-\sin\theta&0\\
 \sin\theta&\cos\theta&0\\
 0&0&1\\
 \end{array}
 \right).
\end{equation}
For the sake of brevity, in what follows, we just use $R$ to denote
the rotational matrix.

Here we consider a hyperelastic structure such as the SVTK material
for which the deformation can be still small while a large rotation
occurs. The key idea is that we can decompose the structure motion
into the rotational part and the deformation part
\cite{Veubeke.D;Fraeijs.B1976a,Felippa.C;Haugen.B2005a}. Suppose
${{\hat\bx}}_0$ is a point on the axis of rotation,
we can decompose the trajectory of each material particle as follows
by using the rotational matrix $R$ arising from the boundary
condition (\ref{eq:displacement_bc}),
\begin{equation}\label{position}
{\bm x}-{{\hat\bx}}_0=R({{\hat\bx}}+\hat{\bm u}_{d}-{{\hat\bx}}_0),
\end{equation}
where $\hat{\bm u}_{d}$ is the local deformation displacement in the
reference configuration, it does not contain any information of the
rotation, thus $\hat{\bm u}_{d}=0$ on $\Gamma_{in}$. Hence, the
total structure displacement, $\hat{\bm u}_s$, can be decomposed as
\begin{equation}\label{displacement}
\hat{\bm u}_s={\bm x}-{{\hat\bx}}=(R-{\bm I})({{\hat\bx}}-\hat{\bm
x}_0)+R\hat{\bm u}_{d}=\hat{\bm u}_\theta+R\hat{\bm u}_{d},
\end{equation}
where $\hat{\bm u}_\theta=(R-{\bm I})({{\hat\bx}}-{{\hat\bx}}_0)$ is
referred to as the rotational part of the structure displacement.

From (\ref{position}) we obtain the deformation gradient tensor $
{\bm F}= R({\bm I}+{\mathbf H}), \text{ where } {\mathbf H}=\nabla
\hat{\bm u}_d $.
Then, the Green-Lagrangian finite strain tensor $ {\bE}=({\bm
F}^T{\bm F}-{\bm I})/2=({\mathbf H}+{\mathbf H}^T+{\mathbf
H}^T{\mathbf H})/2. $
Thus the first Piola-Kirchhoff stress in (\ref{solid-eq-new}),
$\bP_s={\bm F}(\lambda_s (tr{\bE})\bI+2\mu_s{\bE})$, leads to a
function of ${\mathbf H}$ as follows
$$
\bP_s({\mathbf H})=R({\bm
I}+{\mathbf H})\left(\frac{\lambda_s}{2}tr({\mathbf H}+{\mathbf H}^T+{\mathbf H}^T{\mathbf H})\bI+\mu_s({\mathbf H}+{\mathbf H}^T+{\mathbf H}^T{\mathbf H})\right).
$$
Since we only consider a small deformation, i.e. ${\mathbf H}\approx
0$, we then conduct a linear approximation for the above equation by
Taylor expansion around the non-deformable configuration ${\mathbf
H}=0$, resulting in
\begin{equation}
\label{eq:linearize_P} \bP_s\approx
\frac{1}{2}\lambda_sR((tr{\mathbf H})\bI+(tr{\mathbf
H}^T)\bI)+\mu_sR({\mathbf H}+{\mathbf H}^T)=R(\lambda_s
tr(\epsilon(\hat{\bm u}_d)){\bm I}+2\mu_s\epsilon(\hat{\bm u}_d)),
\end{equation}
where $\epsilon(\hat{\bm u}_d)=(\nabla \hat{\bm u}_d+(\nabla
\hat{\bm u}_d)^T)/2$ is the linear approximation of ${\bE}(\hat{\bm
u}_d)$. Therefore, we attain the linear approximation for the stress
tensor in Lagrangian description,
$$ \hat{\bm \sigma}_s=\bP_s\approx
R(\lambda_s tr(\epsilon(\hat{\bm u}_d)){\bm
I}+2\mu_s\epsilon(\hat{\bm u}_d)).
$$

We may equivalently rewrite $\hat{\bm \sigma}_s=R{\bm
D}\epsilon(\hat{\bm u}_d)$, where ${\bm
D}_{ijkl}=2\mu_s\delta_{ik}\delta_{jl}+\lambda_s\delta_{ij}\delta_{kl}$
is a fourth order tensor and ${\bm D}\epsilon(\hat{\bm u}_d)$ is the
contraction of ${\bm D}$ and $\epsilon(\hat{\bm u}_d)$, i.e. $({\bm
D}\epsilon(\hat{\bm u}_d))_{ij}=\sum_{kl}{\bm
D}_{ijkl}(\epsilon(\hat{\bm u}_d))_{kl}.$

Thus, (\ref{solid-eq-new}) is approximated by
\begin{eqnarray}
\hat\rho_s \partial_{tt} \hat{\bm u}_s =\nabla\cdot
\left(R(\lambda_s tr(\epsilon(\hat{\bm u}_d)){\bm
I}+2\mu_s\epsilon(\hat{\bm u}_d))\right)=\nabla\cdot (R{\bm
D}\epsilon(\hat{\bm u}_d)),\label{simplesolid}
\end{eqnarray}
which can be treated as a linear elasticity model of the elastic
rotor, noting that the right hand side of (\ref{simplesolid}) is
linear with respect to $\hat{\bm u}_d$.

Next, we reformulate (\ref{simplesolid}) in terms of an unique
principle unknown $\hat\bu_s$. Due to (\ref{displacement}), we have
$R\hat{\bm u}_d=\hat{\bm u}_s-\hat{\bm u}_\theta$, thus
$\epsilon(\hat{\bm u}_d)=\epsilon(R^T\hat{\bm
u}_s)-\epsilon(R^T\hat{\bm u}_\theta)$. Note that $\hat{\bm
u}_\theta=(R-{\bm I})({{\hat\bx}}-{{\hat\bx}}_0)$, then ${\bm
D}\epsilon(R^T\hat{\bm u}_\theta)={\bm
D}\epsilon((\bI-R^T)(\hat{\bm{x}}-{{\hat\bx}}_0))$. Therefore, we
attain the following equation of the elastic structure in terms of
$\hat{\bm u}_s$ only,
\begin{eqnarray}
\hat\rho_s \partial_{tt} \hat{\bm u}_s -\nabla\cdot (R{\bm
D}\epsilon(R^T\hat{\bm u}_s))&=&-\nabla\cdot (R{\bm
D}\epsilon((R^T-\bI)(\hat{\bm x}_0-\hat{\bm
x}))).\label{simplesolid_us}
\end{eqnarray}

It is assumed that the structure only rotates on the axis of
rotation, i.e., the inner boundary $\Gamma_{in}$. Then, with a given
rotational matrix $R$, the admissible solution set for $\hat{\bm
u}_s$ is defined as follows
$${\mathbf H}_{R}^1(\hat\Omega_s)=\{\hat\bu \in {\mathbf H}^1(\hat\Omega_s)| \hat\bu(\hat{\bm x})=(R-\bI)(\hat{\bm x}-\hat{\bm x}_0), ~\hat{\bm x}\in \Gamma_{in} \}.$$
By (\ref{displacement}), $\hat\bu\in {\mathbf
H}_{R}^1(\hat\Omega_s)$ can be decomposed into the rotational part
$\hat{\bm u}_\theta=(R-{\bm I})({{\hat\bx}}-{{\hat\bx}}_0)$ and the
deformation part $\hat\bu_d\in {\mathbf H}^1_0(\hat\Omega_s)$ such
that
\begin{equation}
\label{eq:decomposition} \hat\bu=R\hat\bu_d+(R-\bI)(\hat{\bm
x}-\hat{\bm x}_0).
\end{equation}
It is obvious that given any $\hat\bu\in {\mathbf
H}_{R}^1(\hat\Omega_s)$, there exists a unique  $\hat\bu_d\in
{\mathbf H}_0^1(\hat\Omega_s)$ such that (\ref{eq:decomposition})
holds.

Assume that $R$ is given for $t\in[0,T]$, then the weak formulation
of (\ref{simplesolid_us}) can be defined as follows, $\forall t\in
[0,T]$, find $\hat{\bm u}_s(t)\in {\mathbf H}_{R}^1(\hat\Omega_s)$
such that $\forall \hat\bphi\in ({H}_0^1(\hat\Omega_s))^d$
\begin{equation}
\noindent(\hat\rho_s \partial_{tt} \hat{\bm
u}_s,\hat\bphi)_{\hat\Omega_s}+ ({\bm D}\epsilon(R^T\hat{\bm
u}_s),\epsilon(R^T\hat\bphi))_{\hat\Omega_s}=({\bm
D}\epsilon((\bI-R^T)({{\hat\bx}}-\hat{\bm
x}_0)),\epsilon(R^T\hat\bphi))_{\hat\Omega_s}
+\int_{\hat\Gamma}\bg_{inf}\hat\bphi\ ds ,\label{newweakform}
\end{equation}
where, the stiffness term on the left hand side of
(\ref{newweakform}) is symmetric positive definite, and $\bg_{inf}$
on the right hand side is the interface force from fluid, which will
be canceled in the weak form of FSI due to the continuity of normal
stress on the interface of fluid and structure.
We conclude that the weak formulation (\ref{newweakform}) can also
be derived based on the linear elasticity equation on the rotational
configuration, as demonstrated at below.

Let $\hat\Omega$ be the reference configuration of elasticity.
Suppose the material undergoes a large rotation along with a small
deformation.  At time $t$, the rotational matrix is given as $R$.
Without loss of generality, let $\hat\bx_0=0$. Define
$$\hat\Omega_R:=\{\hat \bx_R=R\hat \bx, \hat \bx\in \hat\Omega\},$$
which is the rotational configuration. It is known that when
formulating the equation in $\hat\Omega_R$, we can just use the
linear elasticity. When formulating the equation on $\hat\Omega$,
however, the linearization is more complicated. Now we want to show
that we can obtain the same equation by formulating it in different
configurations.

Recall the decomposition of displacement in $\hat\Omega$:
$$\hat\bu_s=(R-\bI)\hat\bx+R\hat\bu_d=\hat\bu_\theta+\hat\bu_R,$$ where,
$\hat\bu_R:=R\hat\bu_d$ is the unknown of the linear elasticity
equation defined in $\hat\Omega_R$. Note that we use the subscript
$R$ for the variables defined in $\hat\Omega_R$, and subscript free
for those variables defined in $\hat\Omega$. To simplify the
notation, we do not distinguish between $f(\hat \bx)$ and $f(R\hat
\bx)$.

First, we look at the stress term in the weak form,
$$
\begin{aligned}
&\left(2\mu_s\epsilon_R(\hat\bu_R)+\lambda_s tr(\epsilon_R(\hat\bu_R))\bI, \nabla_R \bphi\right)_{\hat\Omega_R}\\
=&\left(\mu_s(\nabla_R\hat\bu_R+(\nabla_R\hat\bu_R)^T)+\lambda_s tr(\nabla_R\hat\bu_R+(\nabla_R\hat\bu_R)^T)\bI/2, \nabla_R \bphi\right)_{\hat\Omega_R}\\
=&\left(\mu_s(\nabla \hat\bu_RR^T+(\nabla \hat\bu_RR^T)^T)+\lambda_s tr(\nabla \hat\bu_RR^T+(\nabla \hat\bu_RR^T)^T)\bI/2, \nabla \bphi R^T\right)_{\hat\Omega}\\
=&\left(\mu_s(\nabla \hat\bu_R+R(\nabla \hat\bu_R)^TR)+\lambda_s tr(R^T\nabla \hat\bu_R+(R^T\nabla \hat\bu_R)^T)R/2, \nabla \bphi \right)_{\hat\Omega}\\
=&\left(R[\mu_s(\nabla(R^T \hat\bu_R)+(\nabla (R^T\hat\bu_R))^T)+\lambda_s tr(\nabla (R^T\hat\bu_R)+(\nabla (R^T\hat\bu_R))^T)\bI/2], \nabla \bphi \right)_{\hat\Omega}\\
=&\left(R[2\mu\epsilon(\hat\bu_d)+\lambda_s tr(\epsilon(\hat\bu_d))\bI], \nabla \bphi \right)_{\hat\Omega}.\\
\end{aligned}
$$
This stress term is the same as that shown in (\ref{simplesolid}).
Meanwhile, we have the following inertia term in the rotational
configuration
$$(\hat\rho_s \partial_{tt}\hat\bu_R,\bphi)_{\Omega_R}.$$
Because the rotational configuration is a non-inertial reference
frame, for the sake of force equilibrium, the following centrifugal
force shall be added to the equation
$$(\hat\rho_s\partial_{tt}\hat\bu_{\theta},\bphi)_{\Omega_R}.$$
Thus we can recover the inertia term in (\ref{newweakform}) given
that
$$(\hat\rho_s \partial_{tt}\hat\bu_R,\bphi)_{\Omega_R}+(\hat\rho_s\partial_{tt}\hat\bu_{\theta},\bphi)_{\Omega_R}= (\hat\rho_s\partial_{tt}\hat\bu_{s},\bphi)_{\hat\Omega}.$$
We have, therefore, shown the equivalence with the derivation of the
equation (\ref{newweakform}).

After updating the structure equation with the linear elasticity
involving the rotation, we have the following new weak formulation
of the FSI system with a rotational elastic structure. Find
$\bv_f,\hat\bv_s, p,\cA_{\bu}$, satisfying (\ref{interface}),
(\ref{eq:DirichletBC_f}), (\ref{eq:DirichletBC_sv}), and
(\ref{eq:Initial}), such that for any given $t>0$,  the following
equations hold for $\forall (\hat\bphi,\bphi)\in \mathbb{V},
q\in\mathbb{Q} ,\xi\in  {\mathbf H}^1_0(\hat\Omega_f)$
\begin{equation}\label{ELALE1}
\left\{
\begin{aligned}
(\hat\rho_s\partial_t\hat{\bm v}_s,\hat \bphi)_{\hat\Omega_s}+({\bm
 D}\epsilon(R^T(\hat{\bm u}_s^0+\int_0^t\hat{\bm
v}_s(\tau)d\tau)),\epsilon(R^T\hat\bphi))_{\hat\Omega_s}&+(\rho_f\partial^{\mathcal A}_t{\bm v}_f,\bphi)_{\Omega_f} \\
+(\rho_f({\bm v}_f-\bw)\cdot \nabla {\bm v}_f,\bphi)
_{\Omega_f}+({\bm \sigma}_f,\bm\epsilon(\bphi))_{\Omega_f}&=({\bm
D}\epsilon((\bI-R^T)({{\hat\bx}}-\hat{\bm
x}_0)),\epsilon(R^T\hat\bphi))_{\hat\Omega_s},\\
(\nabla\cdot {\bm v}_f,q)_{\Omega_f}&=0,\\
(\nabla\mathcal{A}_{\bu},\nabla \xi)_{\hat\Omega_f}&=0, \\
\cA_{\bu}&=\hat\bu_s^0+\int_0^t\hat\bv_s(\tau)d\tau,\mbox{ on }\hat\Gamma.\\
\end{aligned}
\right.
\end{equation}

\section{The modified ALE method for an elastic rotor immersed in the fluid}\label{sec:newALE}
In this section, we develop a modified ALE method to deal with the
elastic rotor that is immersed in a non-axisymmetric fluid domain.
Roughly speaking, we first introduce an artificial cylindrical
buffer zone $\Omega_{rf}$ in the fluid domain which embraces the
elastic rotor $\Omega_{s}$ inside and rotates together on the same
axis of rotation and with the same angular velocity, and name the
remaining region in $\Omega_f$ as the stationary fluid subdomain,
denoted by $\Omega_{sf}$, thus
$\Omega_f=\Omega_{sf}\cup\Omega_{rf}$. Denote the interface between
$\Omega_{rf}$ and $\Omega_{sf}$ by $\Gamma_{rs}$, i.e.,
$\Gamma_{rs}=\partial\Omega_{rf}\cap\partial\Omega_{sf}$. On the
other hand, we measure the relative motion of each pair of mesh
nodes on $\Gamma_{rs}$, and add up the difference of motion to the
mesh nodes on $\partial\Omega_{rf}$. Thus, the mesh in
$\bar\Omega_{rf}$ needs to be updated all the time in the numerical
computation in order to conform with the mesh in $\Omega_{s}$
through the interface of $\Omega_{rf}$ and $\Omega_{s}$, $\Gamma$,
and with the mesh in $\Omega_{sf}$ through the interface
$\Gamma_{rs}$ , as shown in Fig. \ref{fig:buffer_fsi_domain}.
\begin{figure}[htbp]
\begin{center}
\includegraphics[width=0.4\textwidth]{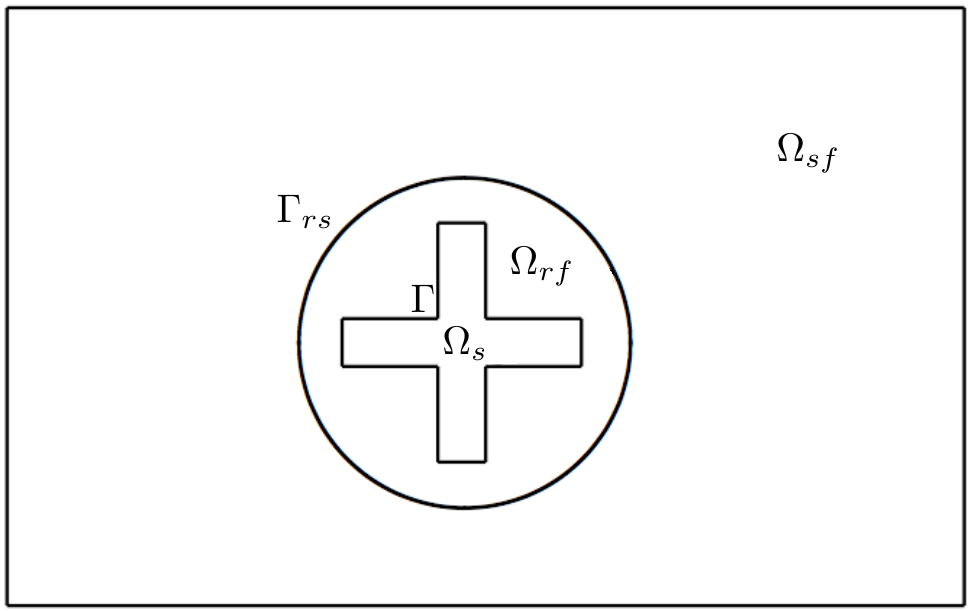}
\end{center}
\caption{The cylindrical buffer zone $\Omega_{rf}$ that embraces an
elastic rotor $\Omega_{s}$} \label{fig:buffer_fsi_domain}
\end{figure}

For the sake of brevity, we use a 2D example to show the idea of how
to use ALE method to generate the fluid mesh in fluid domain
$\Omega_{f}$. Different from the traditional ALE approach, the
modified ALE method only moves the mesh nodes in the rotational
fluid buffer zone $\Omega_{rf}$. We handle the fluid mesh motion in
a similar fashion as we do for the structure, i.e., decomposing the
displacement to the rotational part and the deformation part:
$\hat\bu_s=\hat\bu_\theta+R\hat\bu_d$. In particular, we decompose
$\cA_{\bu}$, the displacement of the rotational fluid mesh defined
in $\hat\Omega_{rf}$, into two parts: the rotational part
$\hat\bu_\theta$ and the deformation part $\cA_D$, i.e., $$\cA_{\bu}
= \hat\bu_\theta + \cA_D,$$ where, $\hat{\bm u}_\theta=(R-{\bm
I})({{\hat\bx}}-{{\hat\bx}}_0)$ is the rotational displacement of
the mesh which is determined by the given rotational matrix, $R$,
defined on the axis of rotation $\Gamma_{in}$. In addition, we also
need the rotational fluid mesh to be conforming with the rotational
and deformable structure mesh from inside, and with the stationary
fluid mesh from outside. The deformation displacement, $\cA_D$,
serves on this purpose by moving the mesh according to the structure
deformation displacement, $\hat\bu_s-\hat\bu_\theta$, on the
interface $\hat\Gamma$, and by locally moving the mesh nodes on the
interface
$\hat\Gamma_{rs}:=\partial\hat\Omega_{rf}\cap\partial\hat\Omega_{sf}$
to properly match with the stationary fluid mesh.
%
%
Namely, we let $\cA_D$ satisfy the following modified ALE mapping
 \begin{equation}\label{ALEmappingLap}
 \left\{
 \begin{aligned}
 -\Delta \mathcal{A}_D&=0, &\mbox{in }\hat\Omega_{rf},\\
 \mathcal{A}_D&= \hat{\bm u}_{s}-\hat{\bm u}_{\theta},& \mbox{on } \hat\Gamma,\\
 \mathcal{A}_D&= \hat{\bm u}_{m},& \mbox{on } \hat\Gamma_{rs},\\
 \end{aligned}
 \right.
 \end{equation}
where, $\hat{\bm u}_m$ needs to be defined such that each sliding
mesh node on $\hat\Gamma_{rs}$ from the side of $\hat\Omega_{rf}$
matches with a fixed mesh node on $\hat\Gamma_{rs}$ from the other
side of $\hat\Omega_{sf}$ by locally moving for such amount of
displacement $\hat{\bm u}_m$ on $\hat\Gamma_{rs}$.

Now using Fig. \ref{fig:interface_match} we illustrate how to find
such small displacement $\hat\bu_m$ in two dimensional case. The
interface is locally shown as a straight line just for an
illustration purpose. In practice, the curved interface is
approximated by piecewise line segments in 2D or piecewise
triangular slice in 3D. The technique shown here is applicable to
the general cases with piecewise line segments/triangular slices.

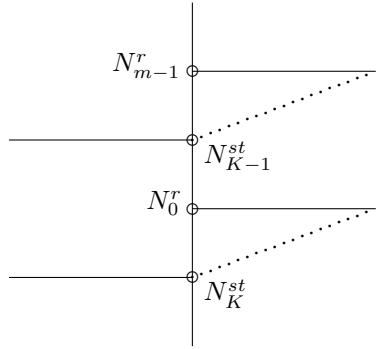
\begin{figure}[h]
\setlength{\unitlength}{0.12in} 
\centering
\begin{picture}(30,16) 
\put(10,0){\line(0,1){15}}
\put(10,3){\line(-1,0){8}}
\put(10,6){\line(1,0){8}}
\put(10,9){\line(-1,0){8}}
\put(10,12){\line(1,0){8}}

\put(10,3){\circle{0.5}}
\put(10,6){\circle{0.5}}
\put(10,9){\circle{0.5}}
\put(10,12){\circle{0.5}}

\put(10.5,2){$N_K^{st}$}
\put(8,6){$N_0^r$}
\put(10.5,8){$N_{K-1}^{st}$}
\put(6.5,12){$N_{m-1}^r$}

\multiput(10 , 3)( 0.4,  0.15) {20} {\circle*{0.1}}
\multiput(10 , 9)( 0.4,  0.15) {20} {\circle*{0.1}}

%
%
%
%
\end{picture}
\caption{Re-matching grid points on the interface of $\Omega_{rf}$}
 \label{fig:interface_match} 
 \end{figure}
Let $\{N_i^r\}_{i=0}^{m-1}$ be coordinates of boundary nodes from
$\partial\Omega_{rf}$ and $\{N_i^{st}\}_{i=0}^{m-1}$ the counterpart
from $\partial\Omega_{sf}$.  After  $\Omega_{rf}$ is rotated, we
assume $N_0^{r}$ falls in between $N_{K-1}^{st}$ and $N_{K}^{st}$.
Then we can choose the following mapping between the nodes:

 $$N_i^r\mapsto N_{[K+i]_m}^{st}.$$

 Here $[m]_n$ denotes the modulo operation. $\hat\bu_m$ can be correspondingly defined on $\{N_i^{r}\}$:
 $$\hat\bu_m(x)= N_{[K+i]_m}^{st}- N_i^r,\quad\mbox{if } x=N_i^r.$$

On $\hat\Gamma\backslash\{N_i^r\}$, $\hat\bu_m$ is defined by linear
interpolation of the nodal values. Since each boundary node from
$\partial\Omega_{rf}$ has a nearby node from $\partial\Omega_{sf}$
to match, the magnitude of  $\hat\bu_m$ is  bounded by mesh
diameter. Therefore it will not cause problems for ALE mapping. Note
that it is also possible to match $N_i^r$ with $N^{st}_{[K+i-1]_m}$.
The choice depends on practical considerations like mesh quality.

Via this approach, the ALE mapping preserves the mesh quality, and,
under the mapping,  the rotational fluid mesh still conforms with
the stationary fluid mesh in $\Omega_{sf}$ through $\Gamma_{rs}$,
and with the rotational structure mesh in $\Omega_{s}$ through
$\Gamma$. The modified ALE method we propose here moves the mesh in
the entire rotational fluid domain $\Omega_{rf}$ to accommodate the
displacement of mesh through the interfaces, thus a better mesh
quality can be attained in a relatively global fashion in
$\Omega_{rf}$ in contrast with the shear-slip method which only
moves the mesh within the shear-slip layer
\cite{Behr.M;Tezduyar.T1999a,Behr.M;Tezduyar.T2001a}.

\begin{remark}
We cannot simply match each of the nodes from $\partial\Omega_{rf}$
with their nearest node from $\partial\Omega_{sf}$ since it may
cause hanging nodes and degenerate elements. We need to move all the
interface nodes clockwise uniformly, or counterclockwise uniformly.
\end{remark}

\begin{remark}
In 3D, the treatment needed for the rotational fluid mesh is much
more sophisticated.  In order to let it rotate, the buffer zone is
chosen as a cylinder when viewed from outside. On the top and bottom
of this cylindrical fluid domain, a natural analogue of two
dimensional case is to generate a mesh that is radially symmetric
about the axis of rotation, thus it can match with the meshes in the
stationary fluid domain by a similar way shown in Figure
\ref{fig:interface_match}. However, the radially symmetric meshes
result in higher aspect ratios of the meshes near the axis of
rotation (See Figure \ref{fig:rotationalFSI1}). Without the radial
symmetry, it may be difficult to let the meshes match on the
interface. This dilemma was addressed in
\cite{Behr.M;Tezduyar.T2001a}.
\end{remark}

\section{Numerical discretizations}
\label{sec:ALEdiscrete} In this section we introduce a full
discretizations for the weak forms of  FSI model. We discretize the
temporal derivative with the characteristic finite difference scheme
(\ref{ALE-MOC}) based upon the moving mesh and on the time interval
$[0,T]$, which is divided into
$$0=t_0<t_1<...<t_N=T.$$ In the space dimension, we adopt the mixed finite element
method to discretize the saddle-point problem (\ref{ELALE1}), where,
we denote $\mathbb{V}_h\subset \mathbb{V}$, $\mathbb{Q}_h\subset
\mathbb{Q}$ and $\mathbb{W}_h\subset H_0^1(\hat\Omega_f)$ for the
finite element spaces of velocity, pressure and the fluid mesh
displacement $\cA_{\bu}$, respectively.

\subsection{A monolithic mixed finite element approximation}
Suppose all the necessary solution data from the last time step,
${\bm v}_f^{n-1}$, $\hat{\bm v}_{s}^{n-1},\ \hat{\bm u}_s^{n-1}$,
$w^{n-1}$, $\mathcal{A}^{n-1}$, are known, and so is the mesh on the
last time step $\mathbb{T}_h^{n-1}$, given as
$\mathbb{T}_h^{n-1}=\mathbb{T}_{sf,h}\cup\mathbb{T}_{rf,h}^{n-1}\cup\hat{\mathbb{T}}_{s,h}$,
where $\hat{\mathbb{T}}_{s,h}$ is the Lagrangian structure mesh in
$\hat\Omega_{s}$ which is always fixed, $\mathbb{T}_{sf,h}$ is the
mesh in the stationary fluid domain $\Omega_{sf}$ which is also
fixed, and $\mathbb{T}_{rf,h}$ is the mesh in the rotational fluid
domain $\Omega_{rf}$ which needs to be computed all the time. In
addition, we discretize (\ref{velocity_displacement}) in the time
interval $[t_n,t_{n+1}]$ with the trapezoidal quadrature rule, as
\begin{equation}\label{u2vtrapzoid}
\hat\bu_s^{n}=\hat\bu_s^{n-1}+\frac{\Delta
t}{2}(\hat\bv_s^n+\hat\bv_s^{n-1}).
\end{equation}
Then a monolithic mixed finite element discretization within a
fixed-point iteration for both fluid and structure equations can be
defined as follows. Let $\bw^{n,0}=\bw^{n-1}$,
$\mathbb{T}_{f,h}^{n,0}=\mathbb{T}_{f,h}^{n-1}$, find ${\bm
v}_f^{n,j},\ \hat{\bm v}_s^{n,j},\ p^{n,j},\cA_{\bu}^{n,j}\ (j=1,\
2,\cdots \text{until convergence})$, satisfying (\ref{interface}),
(\ref{eq:DirichletBC_f}), (\ref{eq:DirichletBC_sv}),
(\ref{eq:Initial}), (\ref{eq:ALE_bc1}) and (\ref{eq:ALE_bc2}),
$\forall (\bphi,\hat\bphi)\in \mathbb{V}_h, q\in\mathbb{Q}_h ,\xi\in
\mathbb{W}_h$, such that
\begin{equation}\label{rotationFSI}
\left\{
\begin{array}{l}
\left(\rho_f\frac{{\bm v}_f^{n,j}-{\bm v}_f^{n-1}}{\Delta
t},\bphi\right)_{\Omega_f^{n,j-1}} +\left(\rho_f({\bm
v}_f^{n,j}-\bw^{n,j-1})\cdot \nabla {\bm
v}_f^{n,j},\bphi\right)_{\Omega_f^{n,j-1}}
+\left(\mu_f\epsilon({\bm v}_f^{n,j}),\epsilon(\bphi)\right)_{\Omega_f^{n,j-1}}\\
-\left(p^{n,j},\nabla\cdot\bphi\right)_{\Omega_f^{n,j-1}}+\left(\rho_s\frac{\hat{\bm
v}_s^{n,j}-\hat{\bm v}_s^{n-1}}{\Delta
t},\hat\bphi\right)_{\hat\Omega_s}+\frac{\Delta
t}{2}\left({\bm D}\epsilon((R^{n})^T\hat{\bm v}_s^{n,j}),\epsilon((R^{n})^T\hat\bphi)\right)_{\hat\Omega_s}=\\
-\frac{\Delta t}{2}\left({\bm D}\epsilon((R^{n})^T\hat{\bm
v}_s^{n-1}),\epsilon((R^{n})^T\hat\bphi)\right)_{\hat\Omega_s}
-\left({\bm D}\epsilon((R^{n})^T\hat{\bm u}_s^{n-1}),\epsilon((R^{n})^T\hat\bphi)\right)_{\hat\Omega_s}\\
+\left({\bm D}\epsilon((I-(R^{n})^T)({{\hat\bx}}-{{\hat\bx}}_0)),\epsilon((R^{n})^T\hat\bphi)\right)_{\hat\Omega_s},\\
\left(\nabla\cdot {\bm v}_f^{n,j},q\right)_{\Omega_f^{n,j-1}}=0,\\
(\nabla\mathcal{A}_{\bu}^{n,j},\nabla \xi)_{\hat\Omega_f}=0,\\
\bw^{n,j}=\frac{\cA_{\bu}^{n,j}-\cA_{\bu}^{n-1}}{\Delta t},\quad
\mathbb{T}_{f,h}^{n,j}=\mathbb{T}_{f,h}^{n-1}+\cA_{\bu}^{n,j}.
\end{array}
\right.
\end{equation}
where, $R^{n}$ is the rotational matrix at time $t_n$ given on the
axis of rotation $\Gamma_{in}$.
To obtain a higher mesh quality from the ALE mapping, the practical
experiences show that a linear anisotropic elasticity equation can
produce more shape-regular mesh with high quality than the harmonic
mapping (\ref{eq:ALE_displacement}) since we can increase the
stiffness of small elements to prevent them from being distorted
\cite{Masud.A1993}.

The only nonlinear term in (\ref{rotationFSI}) is the convection
term $(\rho_f{\bm v}_f^{n,j}\cdot \nabla {\bm v}_f^{n,j},\bphi)$ in
the first equation, which can be linearized by using Newton's method
as
\begin{eqnarray}
(\rho_f{\bm v}_f^{n,j}\cdot \nabla {\bm
v}_f^{n,j},\bphi)\leftarrow(\rho_f{\bm v}_f^{n,j,k-1}\cdot \nabla
{\bm v}_f^{n,j,k},\bphi)+(\rho_f{\bm v}_f^{n,j,k}\cdot \nabla
{\bm v}_f^{n,j,k-1},\bphi)\notag\\
-(\rho_f{\bm v}_f^{n,j,k-1}\cdot \nabla {\bm
v}_f^{n,j,k-1},\bphi),\text{ for } k=1,2, \cdots \text{ until
convergence}. \notag
\end{eqnarray}

If we employ $P_1$-$P_1$ type mixed finite elements to discretize
(\ref{rotationFSI}), then the pressure stabilization term, $\delta
h^2\left(\nabla p^{n,j}, \nabla q\right)$, which was originally
derived from the Galerkin/least-square scheme
\cite{Tezduyar.T1992,Brezzi.F;Pitkaranta.J1984,Hughes.T;Franca.L;Balestra.M1986},
needs to be added to the second equation of (\ref{rotationFSI}),
i.e. the mass equation, in order to obtain a stable solution of
velocity and pressure. The stabilization parameter $\delta h^2$ can
be chosen as $\frac{\delta_0 h^2}{\mu_f}$, and $0<\delta_0<1$ is an
appropriately tuned parameter.
%
Comparing with other stable mixed elements such as $P_2$-$P_1$
(Taylor-Hood) element, $P_1$-$P_1$ element with pressure
stabilization can save a great deal of computational cost with an
acceptable accuracy, especially for the realistic three-dimensional
FSI problems.

In addition, if the Reynolds number of fluid is large, then the
convection term in the momentum equation of fluid turns out to be
dominating, the streamline-upwind/Petrov-Galerkin scheme
\cite{Johanson.C;Navert.U1984} may be used to achieve a stable and
convergent solution, namely, an extra term as follows needs to be
added to the momentum equation
$$
\frac{\delta_{SUPG} h}{\|{\bm
v}_f^{n,j}-\bw^{n,j-1}\|_\infty}\left(({\bm
v}_f^{n,j}-\bw^{n,j-1})\cdot \nabla {\bm v}_f^{n,j},({\bm
v}_f^{n,j}-\bw^{n,j-1})\cdot \nabla \bphi\right)
$$
with an appropriately tuned parameter $\delta_{SUPG}$.
\subsection{Well-posedness of the discrete linear system}
If Stokes-stable finite element pairs are used to discretize the
velocity and pressure, then a robust block preconditioners can be
developed following the approach in \cite{Xu.J;Yang.K2015a}. We consider the linearization of (\ref{rotationFSI}) and use the notation convention from
\cite{Xu.J;Yang.K2015a}, where $\bv=(\bv_f,\hat \bv_s)$; namely, the
velocity has two component: the fluid velocity in Eulerian
coordinates and structure velocity in Lagrangian coordinates.  We
assume that the discretized function space $\mathbb{V}_h$ and
$\mathbb{Q}_h$ are Stokes stable.   For the brevity, we also assume
that the stationary fluid domain vanishes, i.e. all of the fluid
domain is the rotational fluid domain. The technique to be shown
here can be applied to the general cases without any essential 
difficulty.

We consider the linearization of the fluid-structure equations of  (\ref{rotationFSI}). Note that the ALE mesh motion is not included in the equations under consideration.  Define
\begin{equation*}
\begin{aligned}
a(\bv,\bphi)=&\frac{1}{\Delta t}(\rho_f\bv_f,\bphi_f)_{\Omega_f}+\frac{1}{\Delta t}(\hat\rho_s\hat\bv_s,\hat\bphi_s)_{\hat\Omega_s}+(\mu_f\epsilon (\bv_f),\epsilon ( \bphi_f))_{\Omega_f}\\
&+\Delta t(D\epsilon
(R^T\hat\bv_s),\epsilon(R^T\hat\bphi_s))_{\hat\Omega_s},\\
c(\bv,\bphi)=& \int_{\Omega_f}\rho_f(\bw\cdot\nabla)\bv_f\cdot\bphi_f+\int_{\Omega_f}\rho_f(\bv_f\cdot\nabla)\bz\cdot\bphi_f,
\end{aligned}
\end{equation*}
 and
$$b(\bv,p)=(\nabla\cdot \bv_f,p)_{\Omega_f}.$$
$a(\cdot,\cdot)$ are the symmetric terms on the velocity field.  $c(\cdot,\cdot)$ are nonsymmetric terms due to linearization. $\bw$ and $\bz$ are given functions based on previous iteration steps. 
Then the momentum and continuity equations in (\ref{rotationFSI}) can be
formulated as follows:

 Find $\bv\in \mathbb{V}_h$ and  $p\in
\mathbb{Q}_h$ such that
\begin{equation}
\label{eq:saddle}
\left\{
\begin{aligned}
&\tilde a(\bv,\bphi)+b(\bphi,p)&=&( \tilde \bg,\bphi),&\forall& \bphi\in \mathbb{V}_h,\\
&b(\bv,q) &=&0,&\forall& q\in \mathbb{Q}_h,\\
\end{aligned}
\right.
\end{equation}
where $\tilde a(\bv,\bphi) = a(\bv,\bphi)+c(\bv,\bphi)$ and $\tilde \bg$ denotes the summation of all the known terms after linearization in
the first equation of (\ref{rotationFSI}).  Note that we rescale the pressure by multiplying $-1$ such that the problem can be formulated as a symmetric saddle point problem. This variational problem
is very similar to those studied in \cite{Xu.J;Yang.K2015a}. The
only difference is the elastic stress term, where $R$ is added due to the
rotation of the structure. However, the well-posedness can still be
proved in a similar way.

It is well-known that (\ref{eq:saddle}) is well-posed if the
 following conditions can be verified \cite{Girault.V;Raviart.P1986a}

\begin{itemize}
\item
\begin{equation}
\label{eq:brezzi_a}
\tilde a(\cdot,\cdot) \mbox{ is bounded and coercive in } \mathbb{Z}_h:=\{\bv\in\mathbb{V}_h| (\nabla\cdot\bv, q)=0, \forall q\in \mathbb{Q}_h\},
\end{equation}
 \item
 \begin{equation}
 \label{eq:brezzi_b}
 \begin{aligned}
 b(\cdot,\cdot)  \mbox{ is bounded }&\mbox{and satisfies the inf-sup condition } \\
 &\inf_{p\in \mathbb{Q}_h}\sup_{\bv\in \mathbb{V}_h}\frac{b(\bv,p)}{\|\bv\|_V\|p\|_Q}\geq \beta>0.\\
 \end{aligned}
 \end{equation}
\end{itemize}
Here $\|\cdot\|_V$ and $\|\cdot\|_Q$ are given norms of
$\mathbb{V}_h$ and $\mathbb{Q}_h$, respectively. Similar to
\cite{Xu.J;Yang.K2015a}, we define
\begin{equation*}
\begin{aligned}
\|\bv\|_V^2&:=a(\bv,\bv)+r\|\mathcal{P}_h\nabla\cdot\bv_f\|_{0,\Omega_f}^2,\quad \forall \bv\in \mathbb{V}_h,\\
\|q\|^2_Q&:=r^{-1}\|q\|^2_0,\quad \forall q\in\mathbb{Q}_h,\\
\end{aligned}
\end{equation*}
where, $\mathcal{P}_h$ represents the $L^2$ projection from
$\mathbb{Q}$ to $\mathbb{Q}_h$, and  $r=\max\{1,\mu_f, \rho_f{\Delta t}^{-1}, \hat\rho_s{\Delta t}^{-1}, {\Delta t}\mu_s,{\Delta t}\lambda_s\}$.

First, we have the following lemma from \cite{Xu.J;Yang.K2015a}. Let
$\mathbb{V}_{h,f}$ be the discrete fluid velocity space, and
$\bx(\hat\bx,t)$ denotes the fluid mesh motion. In fact,
$\bx(\hat\bx,t)$ is the total mesh displacement of the fluid
computed by ALE mapping: $\cA_{\bu}=\hat \bu_\theta+\cA_D$.
\begin{lemma}
\label{lm:infsup_b_fem} Assume that $\bx(\hat\bx,t)$ is continuous
and satisfies
$$
\bx(\hat\bx,t)|_{\tau}\in \mathcal{P}_1, ~~\forall \tau\in
T_h(\hat\Omega_s)~~\mbox{ and } \inf_{\hat\bx\in \hat\Omega_s}\det
(\nabla_{\hat\bx} \bx(\hat\bx,t)) >0,$$ and that the finite element
pair $(\mathbb{V}_{h,f},\mathbb{Q}_h)$ for the fluid variables
satisfies that
 \begin{equation}
 \label{eq:infsup_f}
 \inf_{q\in\mathbb{Q}_h}\sup_{\bv_f\in
\mathbb{V}_{h,f}}\frac{(\nabla\cdot
\bv_f,q)_{\Omega_f}}{\|\bv_f\|_1\|q\|_0}\geq \beta >0.
 \end{equation}
  Then the following inf-sup condition holds
\begin{equation}
\label{eq:infsup_b_fem} \inf_{q\in\mathbb{Q}_h}\sup_{\bv\in
\mathbb{V}_h}\frac{b( \bv,q)}{\|\bv\|_1\|q\|_0}\geq
\frac{\beta}{\alpha_0^{d/2+1}\alpha_1},
\end{equation}
where
\begin{equation}
\label{eq:d0d1} \alpha_0=\max\left\{\sup_{\hat\bx\in
\hat\Gamma}\|(\nabla_{\hat\bx} \bx(\hat\bx,t))\|_{2},1\right\},\quad
\alpha_1=\max\left\{\sup_{\hat\bx\in
\hat\Gamma}\left\{\det(\nabla_{\hat\bx}
\bx(\hat\bx,t))^{-1}\right\}, 1\right\}.
\end{equation}
\end{lemma}


The following theorem shows that (\ref{eq:saddle}) is well-posed under certain conditions.
\begin{theorem}
\label{thm:wellposed_discrete} Assume the following conditions:

\begin{itemize}
\item  Time step size ${\Delta t}$ is small enough such that the following inequality holds:
\begin{equation}
\label{eq:assumption_c}
C_K\left({\Delta t}\rho_f/\mu_f\right)^{1/2}\|\bw\|_\infty +  {\Delta t}\|\nabla\bz\|_\infty\leq c_0<1,
\end{equation}
where $0<c_0<1$ is a constant. Here $C_K$ is a constant such that $\|\nabla\bu_f\|\leq C_K\|\epsilon(\bu_f)\|$ holds for $\forall\bu_f$.  Moreover, $\|\bw\|_{\infty}=\mbox{ess sup}_{\bx\in\Omega_f}\|\bw(\bx)\|_2$,  $\|\nabla\bz\|_{\infty} = \mbox{ess sup}_{\bx\in\Omega_f}\|\nabla \bz(\bx)\|_2$
\item The assumptions in Lemma
\ref{lm:infsup_b_fem} hold.
\item At a given time step $t^n$,
$\alpha_0,\alpha_1$ and $c_0$ are independent of material and discretization
parameters. 
\end{itemize}
Then, under the norms $\|\cdot\|_V$ and $\|\cdot\|_Q$
the variational problem (\ref{eq:saddle}) at the given time step $t^n$ is uniformly well-posed
with respect to material and discretization parameters.
\end{theorem}

%

\begin{proof}
We prove this theorem by verifying the Brezzi's conditions
(\ref{eq:brezzi_a}) and (\ref{eq:brezzi_b}). First we verify the boundedness of
$\tilde a(\cdot,\cdot)$ and its coercivity in $\mathbb{Z}_h$.  By definition, we first have

$$
\begin{aligned}
a(\bv,\bphi)\leq& \|\bv\|_V\|\bphi\|_V, \quad &\forall \bv,\bphi\in\mathbb{V}_h,\\
a(\bv,\bv)\geq & \|\bv\|_V^2,\quad&\forall \bv\in \mathbb{V}_h.\\
\end{aligned}
$$

Based on the first assumption (\ref{eq:assumption_c}), $c(\cdot,\cdot)$ is a small perturbation. This can be shown by the following estimates:

\begin{equation*}
\begin{aligned}
\int_{\Omega_f}\rho_f(\bw\cdot\nabla)\bu_f\cdot \bv_f \leq C_K \left(\frac{{\Delta t}\rho_f}{\mu_f}\right)^{1/2}\|\bw\|_{\infty}\| \bu\|_{V}\|\bv\|_V,\\ 
\end{aligned}
\end{equation*}

\begin{equation*}
\begin{aligned}
\int_{\Omega_f}\rho_f(\bu_f\cdot\nabla)\bz\cdot \bv_f \leq {\Delta t} \|\nabla\bz\|_{\infty}\|\bu\|_V\|\bv\|_V,\\
\end{aligned}
\end{equation*}

\begin{equation}
\label{eq:c_bounded}
c(\bu,\bv)\leq \left(  C_K\left({\Delta t}\rho_f/\mu_f\right)^{1/2}\|\bw\|_\infty +  {\Delta t}\|\nabla\bz\|_\infty\right)\|\bu\|_V\|\bv\|_V\leq c_0\|\bu\|_V\|\bv\|_V.
\end{equation}

Then we have the boundedness and coercivity  of $\tilde a(\bv,\bphi)=a(\bv,\bphi)+c(\bv,\bphi)$:

\begin{equation}
\label{eq:ac_Velliptic}
\begin{aligned}
\tilde a(\bv,\bv)\geq& (1-c_0)\|\bv\|_V^2,\quad&\forall \bv\in \mathbb{Z}_h,\\
\tilde a(\bv,\bphi)\leq& (1+c_0)\|\bv\|_V\|\bphi\|_V ,\quad&\forall\bv,\bphi\in \mathbb{V}_h.
\end{aligned}
\end{equation}

 The boundedness of $b(\cdot,\cdot)$ is shown by
 \begin{equation}
 \label{eq:b_bounded}
 b(\bv,q)\leq \|\mathcal{P}_h\nabla\cdot\bv\|_{0,\Omega_f}\|q\|_0\leq  r^{1/2}\|\mathcal{P}_h\nabla\cdot\bv\|_{0,\Omega_f}r^{-1/2}\|q\|_0 \leq\|\bv\|_V\|q\|_Q.
  \end{equation}
Therefore, we only need to prove the inf-sup condition of $b(\cdot,\cdot)$.

Since we have
$$\|\epsilon(R^T \hat \bv_s)\|^2\leq C \|\nabla(R^T\hat\bv_s)\|^2=C\|\nabla\hat\bv_s\|^2,$$
 the following inequality holds
\begin{equation}
\label{eq:Vr} \| \bv\|_{V}\leq C r^{1/2}\|
\bv\|_{1,\Omega},\quad\forall\bv\in\mathbb{V}_h.
\end{equation}

Lemma \ref{lm:infsup_b_fem} indicates that
$$
\inf_{q\in \mathbb{Q}_h}\sup_{\bv\in \mathbb{V}_h}\frac{(\nabla\cdot
\bv,q)_{\Omega_f}}{\| \bv\|_V\|q\|_Q}\geq
\frac{C}{\alpha_0^{d/2+1}\alpha_1}.$$ Since $\alpha_0$ and $\alpha_1$ are
independent of material and discretization parameters, the inf-sup
constant is uniformly bounded below.  Therefore, we have shown that
(\ref{eq:saddle}) is uniformly well-posed with respect to material and discretization
parameters.
\end{proof}

\begin{remark}
Based on the well-posedness shown above, we can also derive robust
block preconditioners for the coupled fluid-structure equations as
\cite{Xu.J;Yang.K2015a}. However, we do not elaborate on this topic
in this paper, as the mixed finite element we employ for the
numerical experiments in Section \ref{sec:numerics} is the
$P_1$-$P_1$ type with the pressure stabilization in order to save
the huge computational cost in 3D.  It is an unstable finite element
pair without any stabilization. Robust block preconditioners for
$P_1$-$P_1$ discretization of FSI with the pressure stabilization is
part of the future work we will consider. In our numerical
experiments, we adopt some commonly used block preconditioners for
saddle-point problems, more details of which will be shown in the
next section.
\end{remark}

\section{Algorithm description}\label{section:algorithm} In this
section, we describe a monolithic algorithm involving a relaxed
fixed point iteration for the FSI simulation at the current $n$-th
time step. Roughly speaking, we conduct a fixed-point iteration
between the coupled system of fluid-structure equations and the ALE
mapping equation, where, the fluid-structure coupling system is
solved together, i.e., in the monolithic fashion. Basically, we
solve the coupled equations of fluid and structure on a known mesh
${\mathbb{T}}_{f,h}\cup\hat{\mathbb{T}}_{s,h}$ by a nonlinear
iteration until convergence, where, ${\mathbb{T}}_{f,h}$ is the
fluid mesh updated from the previous iteration step, and
${\mathbb{T}}_{s,h}$ is the fixed structure mesh. Then, we update
the fluid mesh ${\mathbb{T}}_{f,h}$ by solving ALE mapping based
upon the new solution of structure velocity on the interface
$\Gamma$, and so on. Continue this fixed-point iteration until the
fluid mesh converges. Then we march to the next time step.
Details are shown in Algorithm 1, 
where the rotation matrix $R$, depending on the time only, is given
if the angular velocity is prescribed on the axis of rotation for
the case of active rotation. Whereas, in the case of passive
rotation, $R$ is unknown and depends on both space and time, needs
to be updated at each iteration step. We leave the passive rotation
case as another part of our future work.
\begin{algorithm}[H]\label{alg:FSI}
\SetKw{utl}{until} \caption{ALE method for FSI involving an elastic
rotor} On the $n$-th time step, let $\bw^{n,0}=\bw^{n-1}$,
$\mathbb{T}_{f,h}^{n,0}=\mathbb{T}_{f,h}^{n-1}$.\\
\For{$j=1,2,\cdots$ \utl convergence}{
\begin{enumerate}
\item Solve (\ref{rotationFSI}) for $({\bm v}_f^{n,j}, p^{n,j}, \hat{\bm
v}_s^{n,j})$.
\item Compute $\hat{\bm u}_s^{n,j}$ with (\ref{u2vtrapzoid}).
\item Compute $\hat{\bm u}_\theta^{n,j}=(R^{n}-{\bm I})({{\hat\bx}}-{{\hat\bx}}_0)$.
\item Find
$\hat{\bm u}_m^{n,j}$ on $\Gamma_{rs}$ by the searching scheme shown
in Section \ref{sec:newALE}.
\item Update the displacements on the interface:
\begin{eqnarray*}
\hat{\bm u}_{s,*}^{n,j}=(1-\omega)\hat{\bm
u}_{s}^{n,j-1}+\omega\hat{\bm u}_{s}^{n,j},\quad\mbox{on } \hat\Gamma,\notag\\
\hat{\bm u}_{\theta,*}^{n,j}=(1-\omega)\hat{\bm
u}_\theta^{n,j-1}+\omega\hat{\bm u}_\theta^{n,j},\quad\mbox{on }
\hat\Gamma,
\end{eqnarray*}
where, $\omega\in(0,1]$ is the relaxation number.
\item Solve ALE mapping equation (\ref{ALEmappingLap}) for $\mathcal{A}_D^{n,j}$ with
$\hat{\bm u}_{s,*}^{n,j}-\hat{\bm u}_{\theta,*}^{n,j}$  and
$\hat{\bm u}_m^{n,j}$ as the Dirichlet boundary conditions on
$\hat\Gamma$ and $\hat{\Gamma}_{rs}$, respectively.
\item Calculate the total fluid mesh displacement $\hat{\bm u}_f^{n,j}$
with $\hat{\bm u}_f^{n,j}=\hat{\bm
u}_{\theta,*}^{n,j}+\mathcal{A}_D^{n,j}$, with which the rotational fluid mesh, $\mathbb{T}_{rf,h}^{n,j}$, is updated as
$\mathbb{T}_{rf,h}^{n,j}=\mathbb{T}_{rf,h}^{n-1}+\hat{\bm u}_f^{n,j}$.\\
\item Update the fluid mesh velocity $\bw^{n,j}$ by
$\bw^{n,j}=\frac{\hat{\bm u}_f^{n,j}-\hat{\bm u}_f^{n-1}}{\Delta
t}$.
\item
Check if the iteration converges. If not, let $j\leftarrow j+1$, and
continue the iteration. If yes, let $n\leftarrow n+1$, and march to
the next time step.
\end{enumerate}
}
\end{algorithm}

In Algorithm 1, the most time consuming part is to solve
(\ref{rotationFSI}), an efficient and robust linear solver is
crucial to speed up the simulation. In the following we briefly
introduce the preconditioning technique we used in our solver. For
the convenience of notation, we denote the linear system of the
coupled fluid-structure system as
\begin{equation}
\left(
\begin{array}{cc}
A&B^T\\
B&-C
\end{array}
\right)
\left(
\begin{array}{c}
v\\
p\\
\end{array}
\right)
=
\left(
\begin{array}{c}
f\\
0\\
\end{array}
\right),
\end{equation}
where, $v$ and $p$ are the vectors corresponding to velocity and
pressure, respectively. The block matrices $A$ and $B$ arise from
the corresponding discretizations of the bilinear forms
$\tilde a(\cdot,\cdot)$ and $b(\cdot,\cdot)$, respectively. The nonzero
block $C$ is from the pressure stabilization if the $P_1$-$P_1$
mixed element is used, otherwise $C=\emptyset$.

We use the flexible GMRes as the iterative solver for the coupled
system with a block triangular preconditioner $P$, which is an
approximation of

$$\left(
\begin{array}{cc}
A&0\\
B&-S
\end{array}
\right)^{-1},
$$
where $S:= C+B\mbox{diag}(A)^{-1}B^T$ is an approximation of the
Schur complement $C+BA^{-1}B^T$.  We define the action of
$\mathbb{P}$ as follows.  Given $\binom{f}{g}$, we obtain
$\binom{v}{p}=\mathbb{P}\binom{f}{g}$ by the following procedure.
\begin{enumerate}
\item Solve $Av=f$ using the algebraic multigrid preconditioned GMRes. In particular, we use the unsmoothed aggregation AMG with ILU and Gauss-Seidel smoothers.
\item Solve $Sp=-g+Bv$ using the algebraic multigrid preconditioned GMRes. Here we use the unsmoothed aggregation AMG with the Gauss-Seidel smoother.
\end{enumerate}

There are extensive literature on fast solvers for the monolithic
FSI simulation. We refer to
\cite{Heil.M2004a,Badia.S;Quaini.a;Quarteroni.a2008b,Badia.S;Quaini.A;Quarteroni.A2008a,Gee.M;Kuttler.U;Wall.W2011a,Turek.S;Hron.J2010a,Cai.X2010a,Barker.A;Cai.X2009a,Barker.A;Cai.X2010a,Barker.A;Cai.X2010b,Wu.Y;Cai.X2014a,Xu.J;Yang.K2015a}
for details. In particular, when stable mixed finite element pairs
are used, we can develop robust block preconditioners based on the
well-posedness shown in Theorem \ref{thm:wellposed_discrete}. (See
\cite{Xu.J;Yang.K2015a}.)

\begin{remark}
Because of the simultaneous solution of both fluid and structure in
the monolithic method, the interface conditions are naturally
enforced.  Therefore, the resulting algorithm is much more stable
than partitioned algorithms. Especially, the partitioned method
turns out to be unconditionally unstable with an explicit scheme, or
to be a oscillating iteration with questionable convergence when an
implicit scheme is used, if the densities of fluid and structure
become comparable or if the domain has a slender shape. Such
phenomenon is also called the added-mass effect
\cite{Richter.T;Wick.T2010a,Ryzhakov.P;Rossi.R;Idelsohn.S;Onate.E2010a,
Idelsohn.S;Del.P;Rossi.R;Onate.E2009a,Hou.G;Wang.J;Layton.A2012a},
which the monolithic algorithm can completely avoid.
\end{remark}

\section{Numerical experiments}
\label{sec:numerics} To verify the correctness and efficiency of our
proposed model and numerical methods, we first introduce a
simplified three-dimensional hydro-turbine with the shape of a
curving cross immersed in the laminar flow which is fully developed
in a straight open channel. As shown in Fig. \ref{fig:toy}, the
cross lies close to the inlet and rotates about its rigid axis of
rotation in the plane. Its axis of rotation is parallel with the
flow direction, thus it faces to the incoming flow and may deform to
some extent due to the fluid impact.
The geometrical and physical parameters of this model problem are
given in Table \ref{table:2dcross_parameters}. The hydro-turbine
rotor is simulated under the prescribed steady incoming flow, which
is described as a parabolic velocity function at the inlet with the
maximum value of 1.5 m/s, and the prescribed angular velocity at the
axis of rotation of rotor $\Gamma_{in}$,  of 1 rad/s.  The incoming
flow at inlet and the angular velocity at $\Gamma_{in}$  are
responsible for the Dirichlet boundary conditions of fluid velocity
and structure velocity, respectively. In addition, the no-slip
boundary condition is applied to the wall, and do-nothing boundary
condition is applied to the fluid velocity at the outlet.
\begin{table}[!htb]
\caption{Geometrical, physical and operating parameters
\label{table:2dcross_parameters}}\centerline{
\begin{tabular}{llll}
  \hline
  $Parameters/properties$ & $Symbol$ & $Value$ & $Unit$\\
  \hline
  \underline{\emph{Modeling domain dimensions}} &&&\\
  Channel length &$L_{\text{channel}}$    & $0.5$&m\\
  Channel width (height)  &$W_{\text{channel}}$    & $0.2$&m\\
  Cross length &$L_{\text{cross}}$ &0.1&m\\
  Cross width &$W_{\text{cross}}$ &0.015&m\\
  Cross thickness &$T_{\text{cross}}$ &0.05&m\\
  \hline
  \underline{\emph{Physical and transport parameters}} &&&\\
  Dynamic viscosity of fluid &$\mu_f$  &1.0&kg/(m$\cdot$s)\\
  Density of fluid &$\rho_f$&1000&kg/m$^3$\\
  Density of structure &$\rho_s$&1280&kg/m$^3$\\
  Young's modulus &$E$&$2.5\times 10^6$&Pa\\
  Poisson's ratio &$\nu$&0.384&\\
  \hline
 \underline{\emph{Operating parameters}} &&&\\
 The maximum incoming velocity & ${\bm u}_{\text{in}}$ & 1.5 & m/s\\
 The prescribed angular velocity & $\omega$ & 1.0 & rad/s\\
 Tolerance of nonlinear iteration & $\varepsilon_{tol}$& $10^{-6}$ &\\
 \hline
\end{tabular}}
\end{table}


\begin{figure}[htbp]
\begin{center}
\graphicspath{{:figure:}}
\includegraphics*[height=1.4in,width=2.2in]{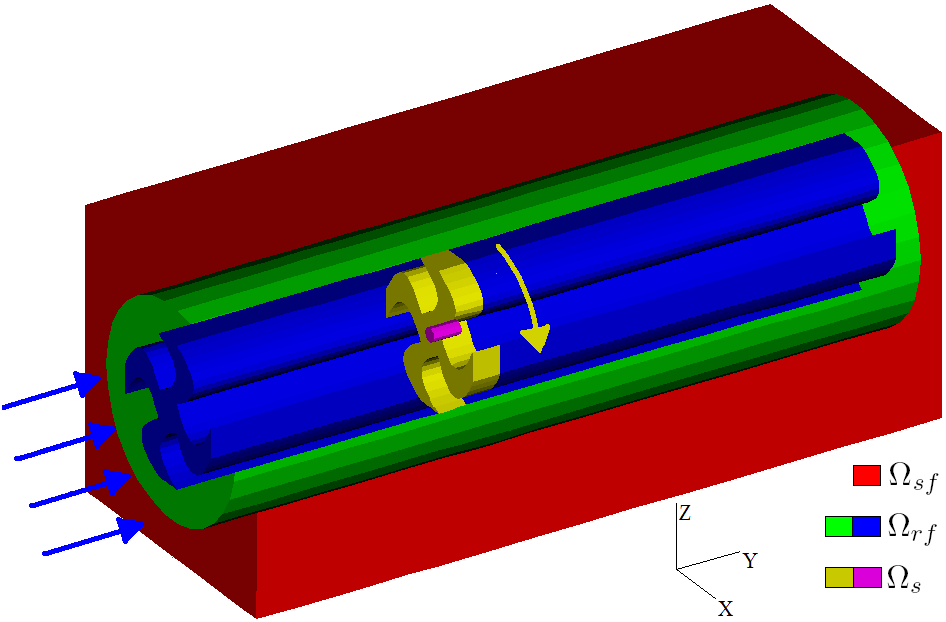}
\hspace*{0.5in}
\includegraphics*[height=1.4in,width=2.in]{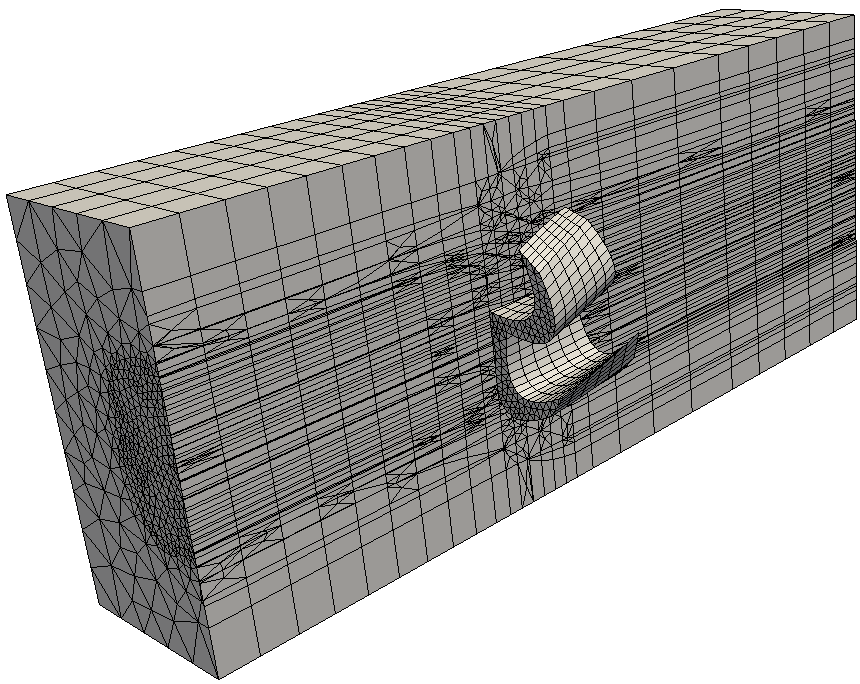}
\end{center}
\caption{The 3D self defined elastic rotor:
$\Omega=\Omega_s+\Omega_f$ and $\Omega_f=\Omega_{sf}+\Omega_{rf}$,
and the mesh of half domain, where the interfaces
$\Gamma_{rs}=\partial\Omega_{rf}\cap\partial\Omega_{sf}$ and
$\Gamma=\partial\Omega_{rf}\cap\partial\Omega_{s}$.} \label{fig:toy}
\end{figure}

As illustrated by Fig. \ref{fig:toy}, we define a cylindrical fluid
subdomain $\Omega_{rf}\subset\Omega_f$ to embrace the elastic rotor
$\Omega_s$ and rotate together on the same axis of rotation. Let
$\Omega_{sf}:=\Omega_f\backslash \Omega_{rf}$ denote the rest of the
fluid domain that is stationary. Moreover, on the interface between
$\Omega_{sf}$ and $\Omega_{rf}$, $\Gamma_{sf}$, and the interface
between $\Omega_{rf}$ and $\Omega_s$, $\Gamma$, we apply the
Master-Slave relations to the pairs of grid points along with a
certain triangulation pattern. Such triangulation pattern on the
interface $\Gamma_{rs}$ is defined in a specific manner such that it
is easy to search and reconnect each pair of the grid point when the
slave point on one side ($\partial\Omega_{rf}$) rotates away from
its master point on the other side ($\partial\Omega_{sf}$). For
instance, since we need to choose an axisymmetric domain as the
rotational fluid domain $\Omega_{rf}$, the cross section that is
perpendicular to the central axis of $\Omega_{rf}$ is always a disk,
on which the grid pattern can be triangulated as shown in Fig.
\ref{fig:rotationalFSI1} in order to easily search the paired grid
points on $\Gamma_{rs}$. We basically use the approach shown in Fig.
\ref{fig:interface_match} to re-match each point pair on
$\Gamma_{rs}$ when they are away from each other. Therefore, the
entire fluid mesh retains the conformity.
\begin{figure}[htbp]
\begin{center}
\includegraphics[height=1.5in,width=1.5in]{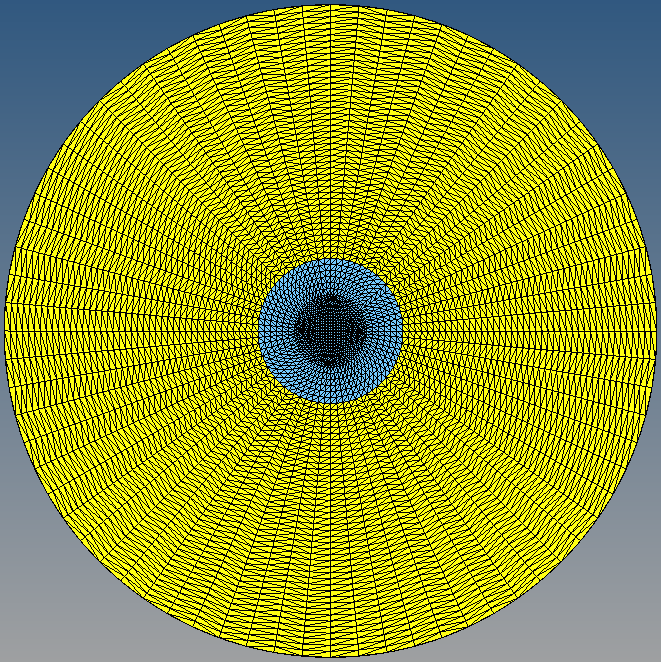}
\hspace*{0.5in}
\includegraphics[height=1.5in,width=1.5in]{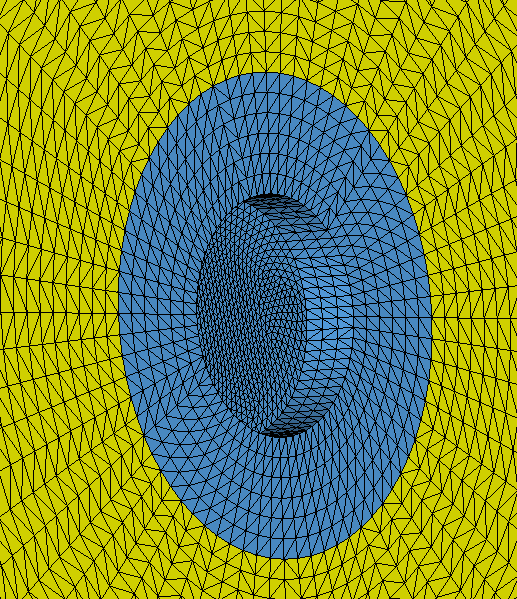}
\end{center}
\caption{The triangulation pattern on the interface of the
rotational fluid domain $\Omega_{rf}$ and the stationary fluid
domain $\Omega_{sf}$.} \label{fig:rotationalFSI1}
\end{figure}

Thereafter, by means of $P_1$-$P_1$ mixed element with pressure
stabilization to discretize momentum and mass equations, linear
finite element to discretize ALE equation, and Algorithm 1, we
obtain the following numerical results with the time step size
$\Delta t=0.01$s and the mesh size $h<0.02$m, as shown in Figs.
\ref{fig:velocitymagnitude}-\ref{fig:cross3dstreamline}.
\begin{figure}[htdp]
\centerline{
\begin{tabular}{c}
\parbox{1.9in}{ \subfigure[Time=2s]{
   \begin{minipage}[t]{1\linewidth}
      \centering
      \includegraphics*[height=.9in,width=1.9in]{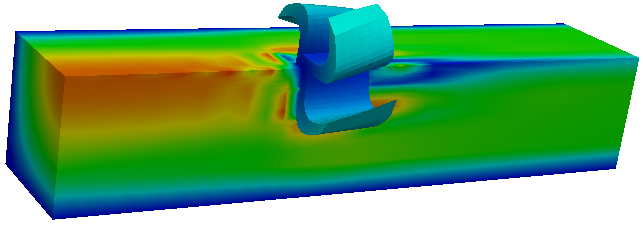}
      \end{minipage}}}
\hspace*{0.0in}
\parbox{1.9in}{\subfigure[Time=4s]
{
   \begin{minipage}[t]{1\linewidth}
      \centering
      \includegraphics*[height=.9in,width=1.9in]{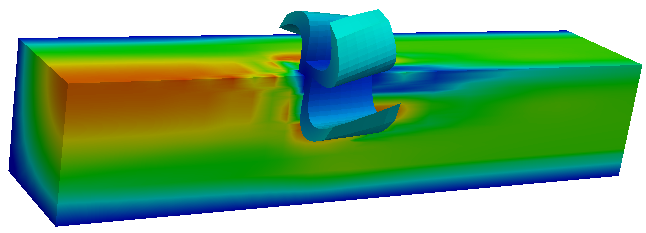}
      \end{minipage}}}
\hspace*{0.0in}
\parbox{1.9in}{\subfigure[Time=6s]
{
   \begin{minipage}[t]{1\linewidth}
      \centering
      \includegraphics*[height=.9in,width=1.9in]{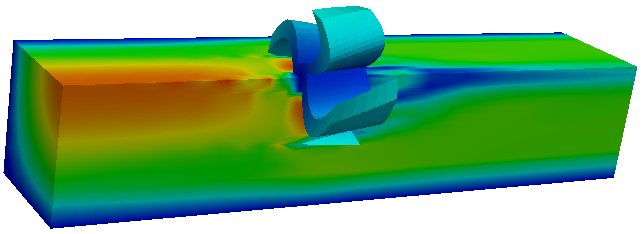}
      \end{minipage}}}
\hspace*{0.0in}
\parbox{.4in}{ 
   \begin{minipage}[t]{1\linewidth}
      \centering
      \includegraphics*[height=1.in,width=.4in]{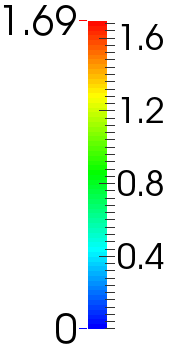}
      \end{minipage}}
\\
\parbox{1.9in}{ \subfigure[Time=8s]{
   \begin{minipage}[t]{1\linewidth}
      \centering
      \includegraphics*[height=.9in,width=1.9in]{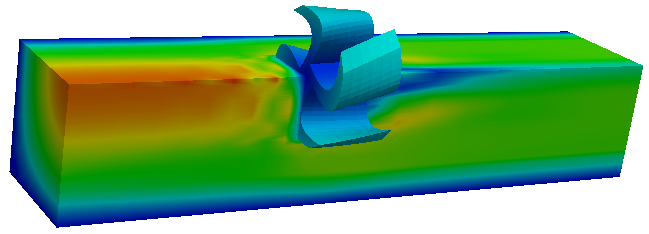}
      \end{minipage}}}
\hspace*{0.0in}
\parbox{1.9in}{\subfigure[Time=10s]
{
   \begin{minipage}[t]{1\linewidth}
      \centering
      \includegraphics*[height=.9in,width=1.9in]{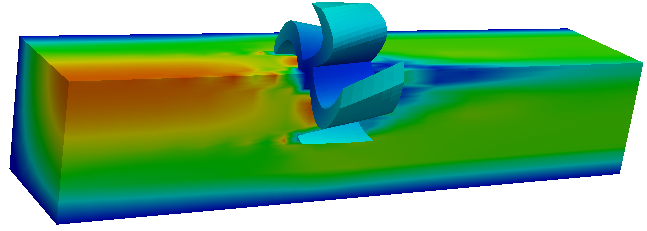}
      \end{minipage}}}
\hspace*{0.0in}
\parbox{1.9in}{\subfigure[Time=12s]
{
   \begin{minipage}[t]{1\linewidth}
      \centering
      \includegraphics*[height=.9in,width=1.9in]{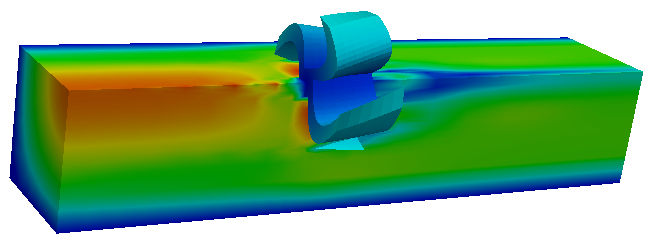}
      \end{minipage}}}
\hspace*{0.0in}
\parbox{.4in}{ 
   \begin{minipage}[t]{1\linewidth}
      \centering
      \includegraphics*[height=1.in,width=.4in]{vmscale.png}
      \end{minipage}}
\\
\parbox{1.9in}{ \subfigure[Time=14s]{
   \begin{minipage}[t]{1\linewidth}
      \centering
      \includegraphics*[height=.9in,width=1.9in]{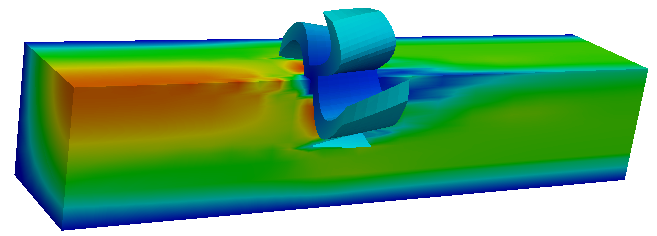}
      \end{minipage}}}
\hspace*{0.0in}
\parbox{1.9in}{\subfigure[Time=16s]
{
   \begin{minipage}[t]{1\linewidth}
      \centering
      \includegraphics*[height=.9in,width=1.9in]{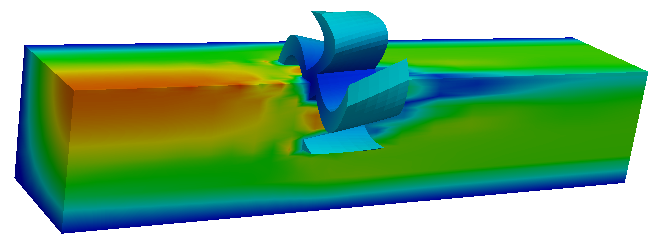}
      \end{minipage}}}
\hspace*{0.0in}
\parbox{1.9in}{\subfigure[Time=18s]
{
   \begin{minipage}[t]{1\linewidth}
      \centering
      \includegraphics*[height=.9in,width=1.9in]{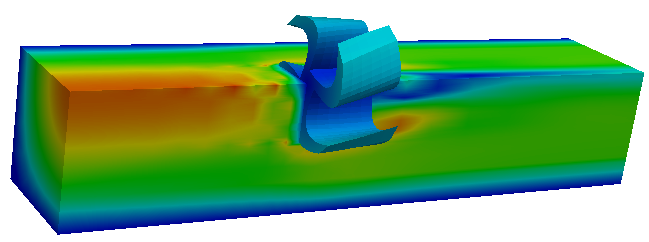}
      \end{minipage}}}
\hspace*{0.0in}
\parbox{.4in}{ 
   \begin{minipage}[t]{1\linewidth}
      \centering
      \includegraphics*[height=1.in,width=.4in]{vmscale.png}
      \end{minipage}}
\end{tabular}
} \caption{A self-defined elastic rotor: the evolution of magnitude
of velocity with time in 1/4 fluid domain.}
\label{fig:velocitymagnitude}
\end{figure}
Fig. \ref{fig:velocitymagnitude} illustrates the evolution of the
magnitude of velocity field with time marching in a quarter part of
the fluid domain along the flow direction, while the incoming fluid
keeps flowing in the inlet with the prescribed maximum velocity
magnitude, 1.5 m/s, and gets detoured by the elastic rotor which
preserves spinning around its axis of rotation in the plane that is
perpendicular to the flow direction with the prescribed angular
velocity, 1 rad/s. With these given steady velocities, we can
observe that the entire fluid field basically attains a relatively
steady state after 2s, the fluid flow is significantly disturbed
near the spinning turbine, such disturbance dies away along with the
increasing distance from the turbine. The magnitude of velocity
field forms a nearly vacuum region behind the turbine in which the
velocity magnitude is much smaller than elsewhere due to the lower
pressure there, and a portion of this region that is nearest to the
turbine is twisted by the spin, periodically. Fig.
\ref{fig:cross3dstreamline} shows the streamline field and velocity
vector field at 19s, respectively, further illustrating the
disturbance status of the fluid right after it flows over the
spinning turbine, where, the streamline is twisted away from the
mainstream near the turbine due to its spin, then merges back again
when the fluid is relatively far away from the turbine, and continue
to flow straight to the outlet.

Note that the magnitude of Reynolds number that we adopt in this
example is just 100 or so, truly resulting in a laminar flow. Since
we only focus on the methodology study of simulating the dynamical
interaction between an elastic rotor and the fluid, the fluids with
a large Reynolds number or even turbulence flow are not our concern
in this paper. In the future, we will consider a turbulence model
for fluids with a large Reynolds number, but the ALE approach and
the monolithic algorithm developed in this paper are still valid for
the simulation of the interaction between the turbulent flow and a
hydro-turbine.

\begin{figure}[htdp]
\centerline{
\parbox{1.7in}{ \subfigure[Magnitude of velocity field]{
   \begin{minipage}[t]{1\linewidth}
      \centering
\includegraphics[height=1.2in,width=1.7in]{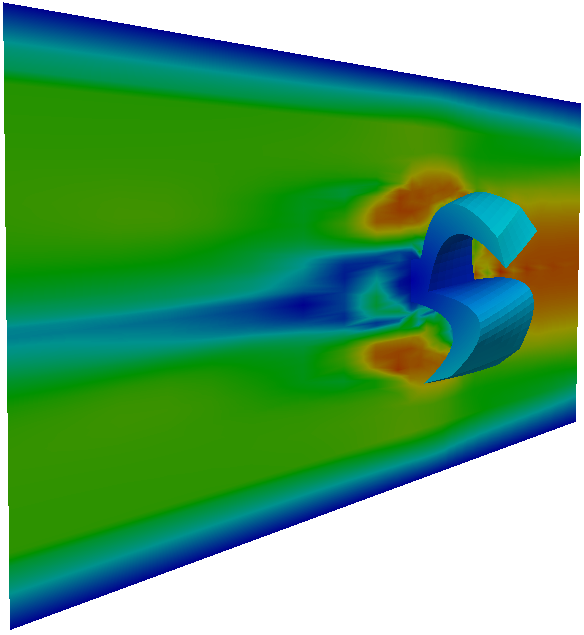}
      \end{minipage}}}
\hspace*{0.0in}
\parbox{1.9in}{ \subfigure[Streamline field]{
   \begin{minipage}[t]{1\linewidth}
      \centering
\includegraphics[height=1.in,width=1.9in]{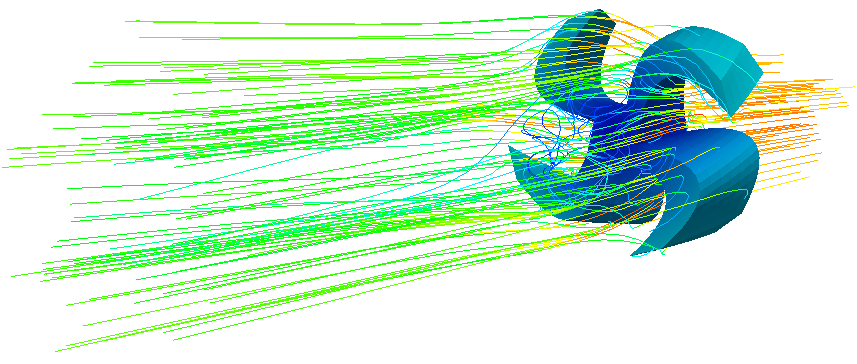}
      \end{minipage}}}
\hspace*{0.0in}
\parbox{1.9in}{\subfigure[Velocity field]
{
   \begin{minipage}[t]{1\linewidth}
      \centering
\includegraphics*[height=1.2in,width=1.9in]{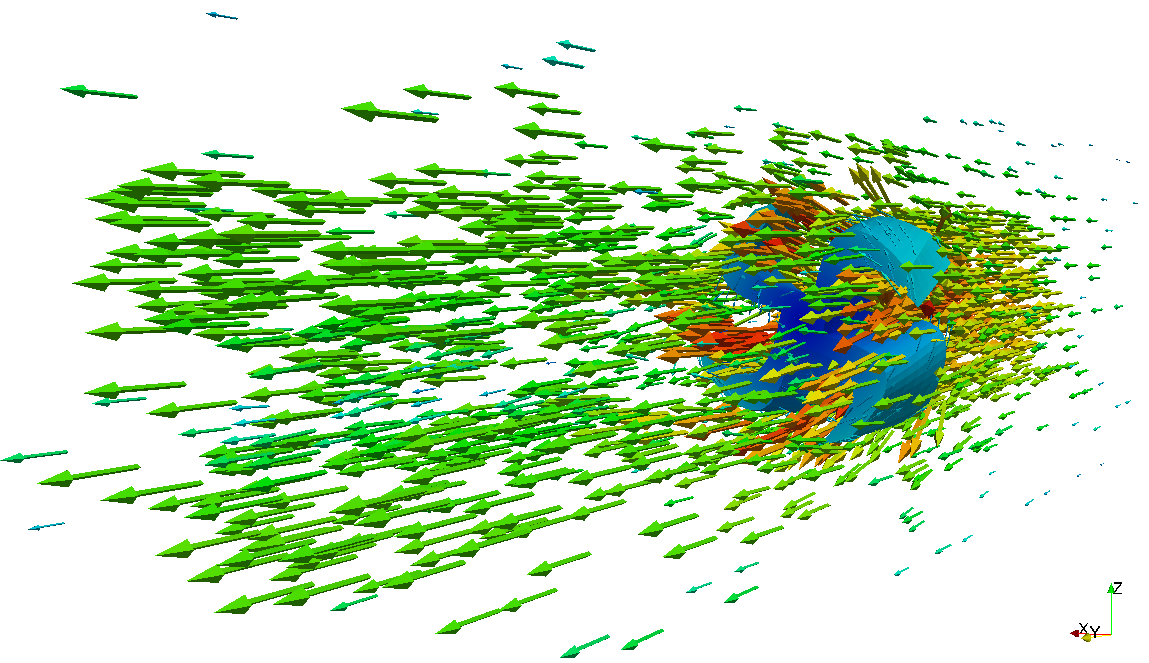}
      \end{minipage}}}
\hspace*{0.0in}
\parbox{.4in}{ 
   \begin{minipage}[t]{1\linewidth}
      \centering
      \includegraphics*[height=1.in,width=.4in]{vmscale.png}
      \end{minipage}}
} \caption{A self-defined elastic rotor: magnitude of velocity field
on the central cross section along flow direction (left); streamline
(middle) and velocity (right) fields at 19s.
\label{fig:cross3dstreamline}}
\end{figure}

It is not easy to observe the deformation of the elastic rotor since
its stiffness is measured by a relatively large Young's modulus (2.5
MPa) in this example. To show the deformation of a spinning rotor
due to the fluid impact, we need to look at some extreme
circumstances, e.g., a flexible and/or a thinner rotor with smaller
Young's modulus, in order to observe a dramatically large
deformation occurring on the hydro-turbine blades while rotating.
Fig. \ref{fig:softsolid} illustrates the desired numerical results,
where the self defined flexible and thinner blades with a relatively
smaller Young's modulus deform while rotate, dramatically.




\begin{figure}[htdp]
\centerline{
\begin{tabular}{c}
\parbox{1.9in}{ \subfigure[Time=0.06s]{
   \begin{minipage}[t]{1\linewidth}
      \centering
      \includegraphics*[height=.9in,width=1.9in]{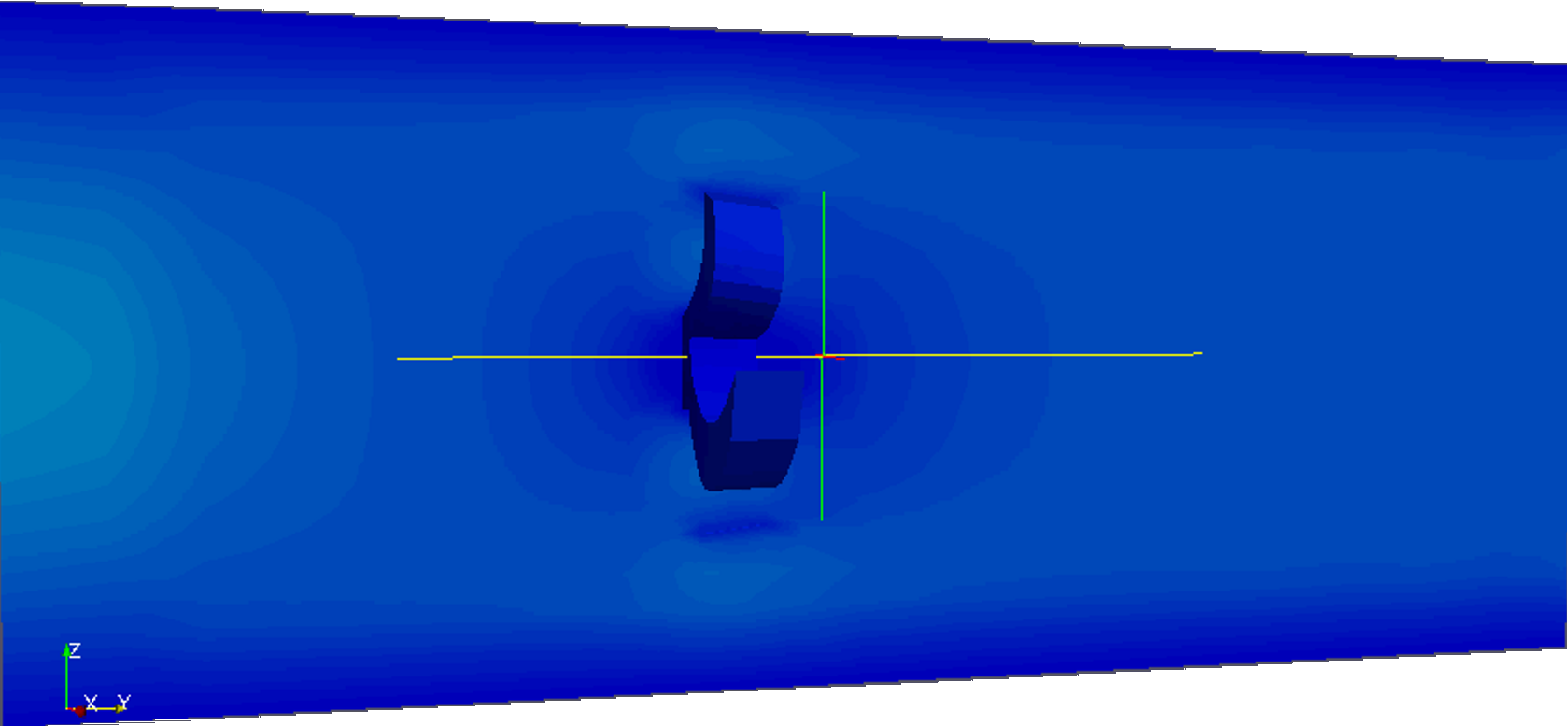}
      \end{minipage}}}
\hspace*{0.0in}
\parbox{1.9in}{\subfigure[Time=0.12s]
{
   \begin{minipage}[t]{1\linewidth}
      \centering
      \includegraphics*[height=.9in,width=1.9in]{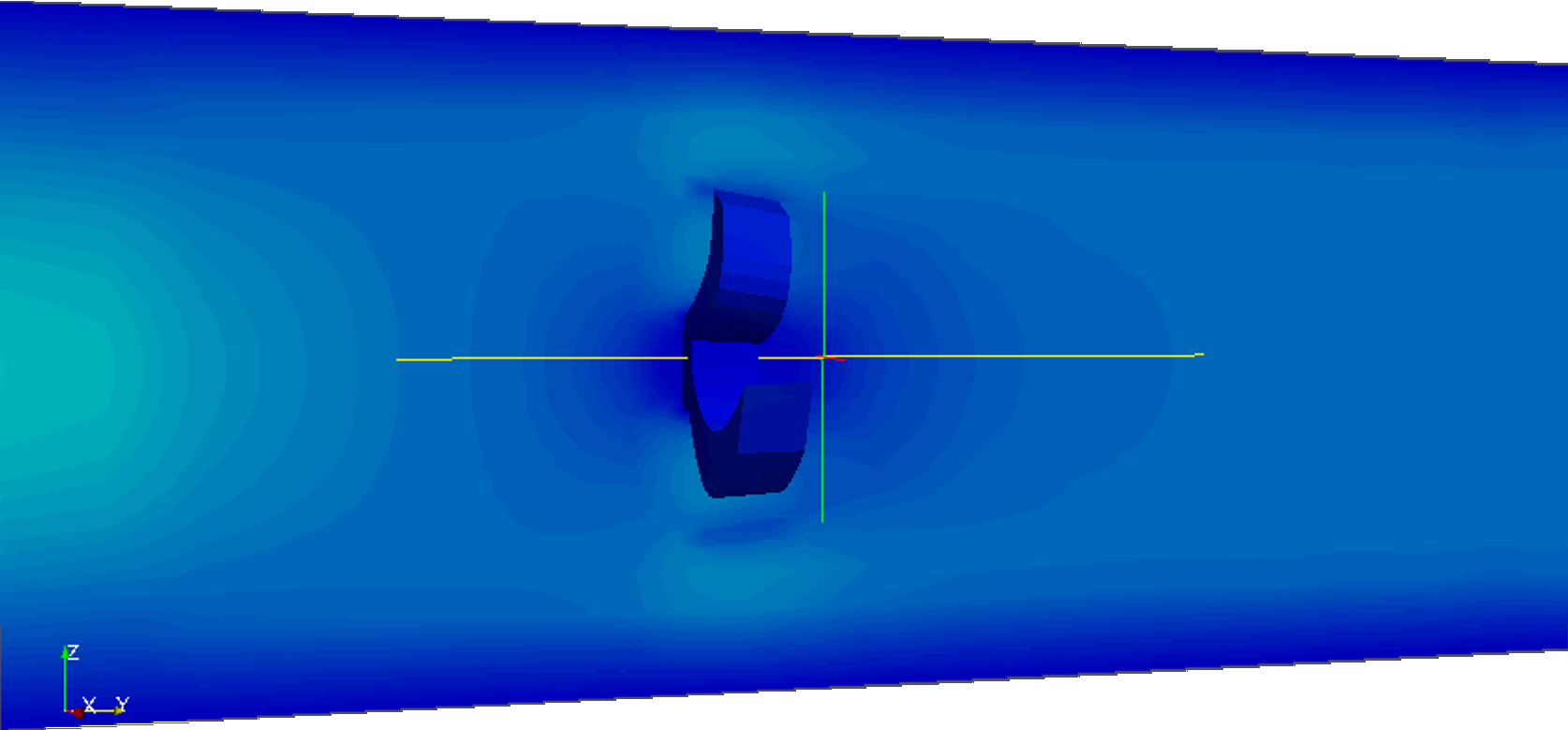}
      \end{minipage}}}
\hspace*{0.0in}
\parbox{1.9in}{\subfigure[Time=0.18s]
{
   \begin{minipage}[t]{1\linewidth}
      \centering
      \includegraphics*[height=.9in,width=1.9in]{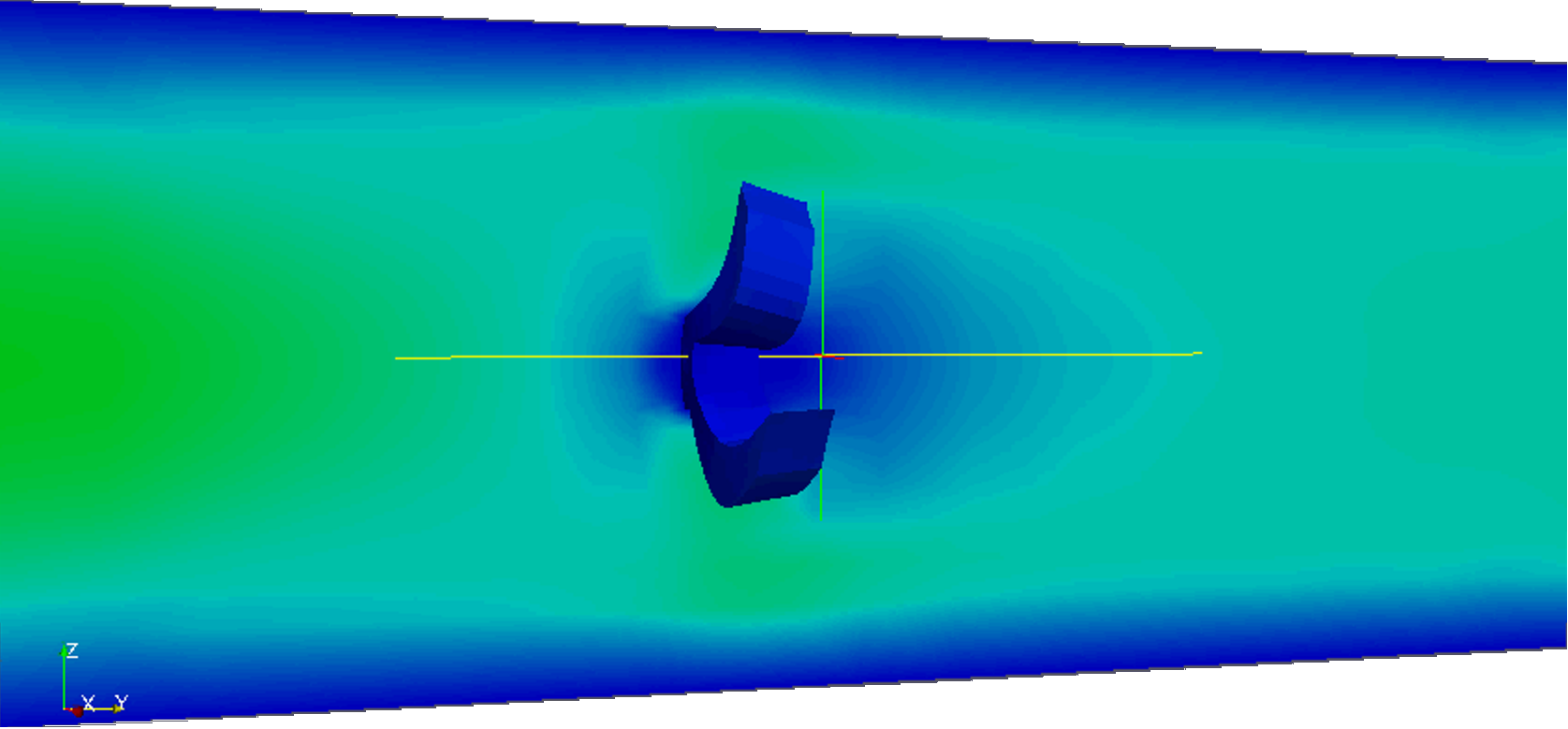}
      \end{minipage}}}
\\
\parbox{1.9in}{ \subfigure[Time=0.24s]{
   \begin{minipage}[t]{1\linewidth}
      \centering
      \includegraphics*[height=.9in,width=1.9in]{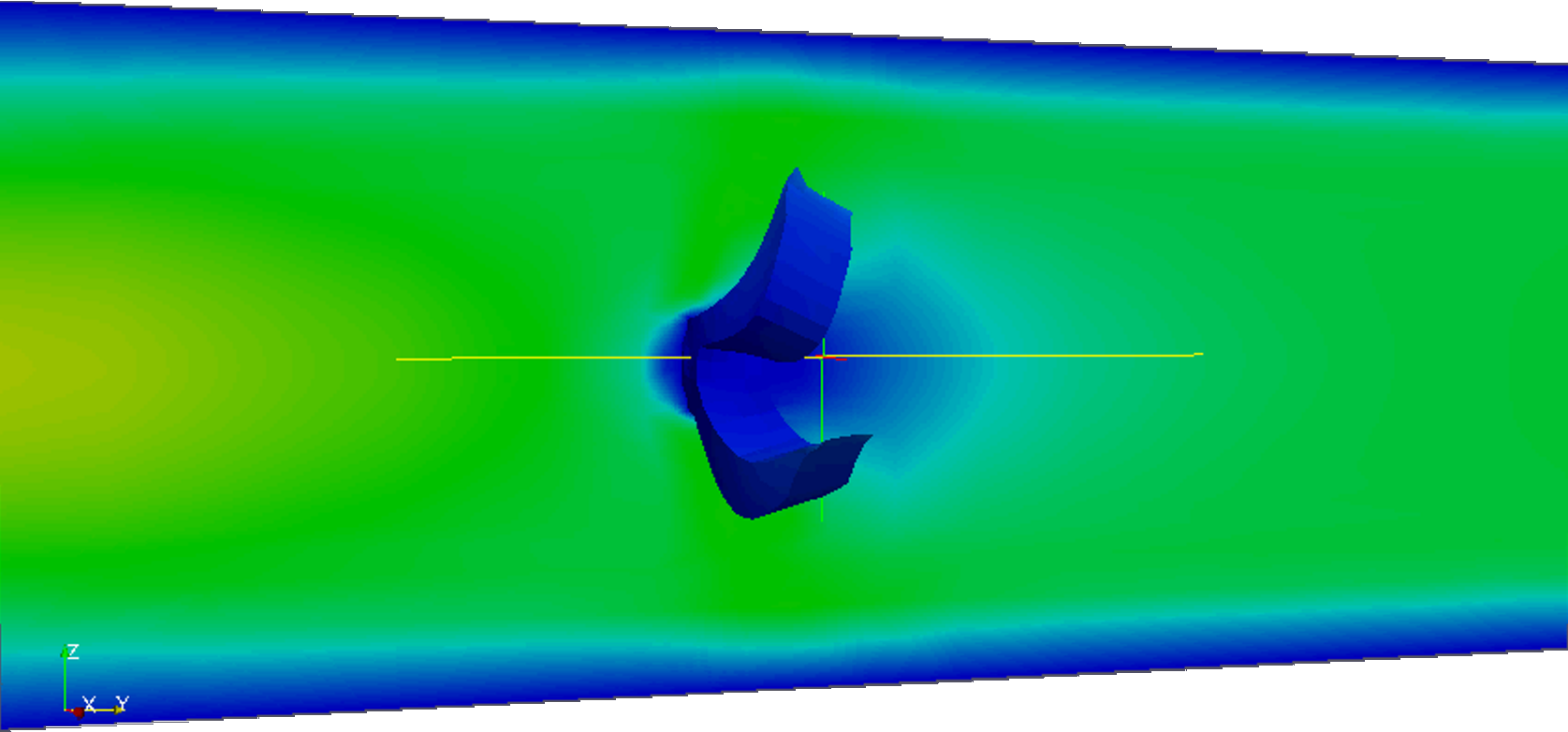}
      \end{minipage}}}
\hspace*{0.0in}
\parbox{1.9in}{\subfigure[Time=0.3s]
{
   \begin{minipage}[t]{1\linewidth}
      \centering
      \includegraphics*[height=.9in,width=1.9in]{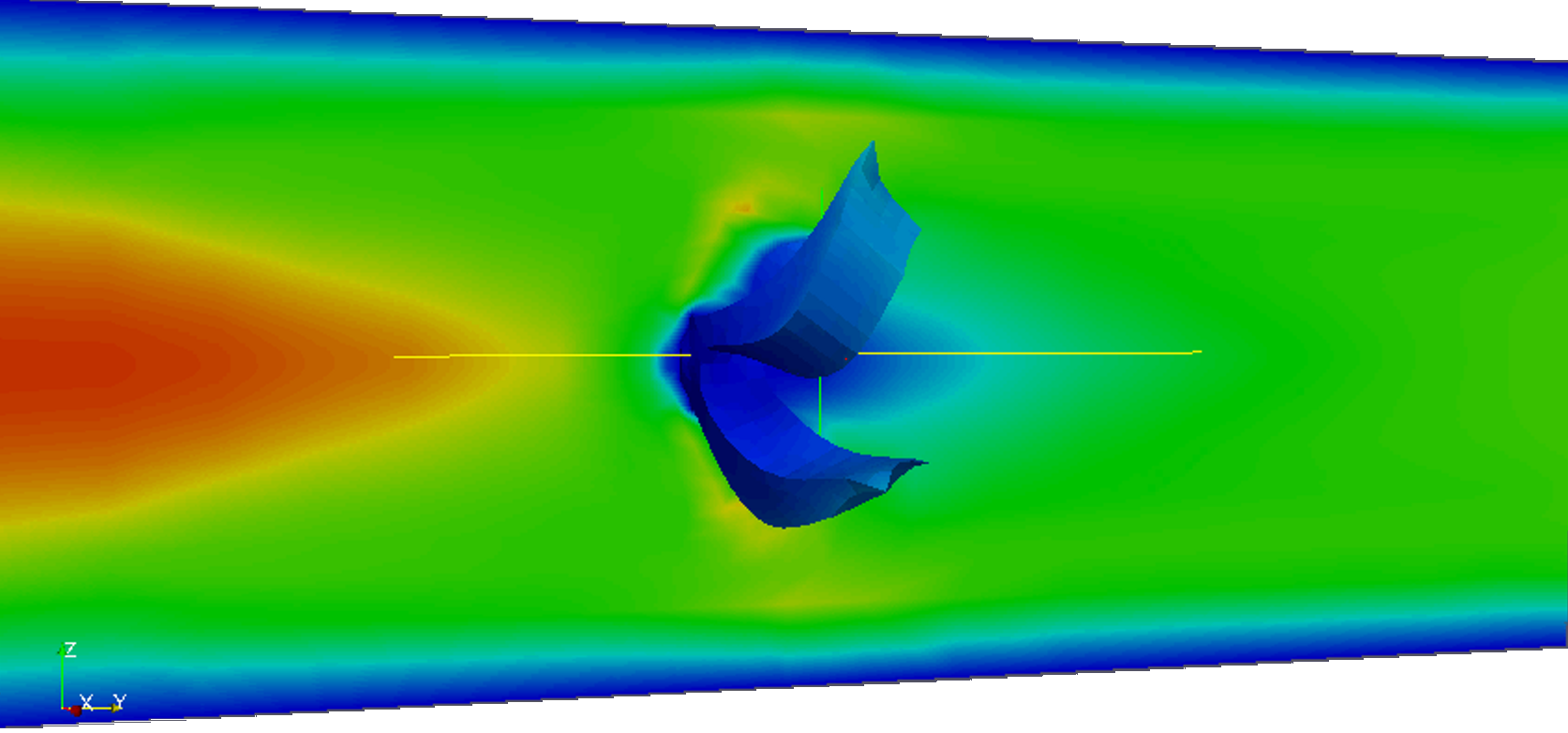}
      \end{minipage}}}
\hspace*{0.0in}
\parbox{1.9in}{\subfigure[Time=0.36s]
{
   \begin{minipage}[t]{1\linewidth}
      \centering
      \includegraphics*[height=.9in,width=1.9in]{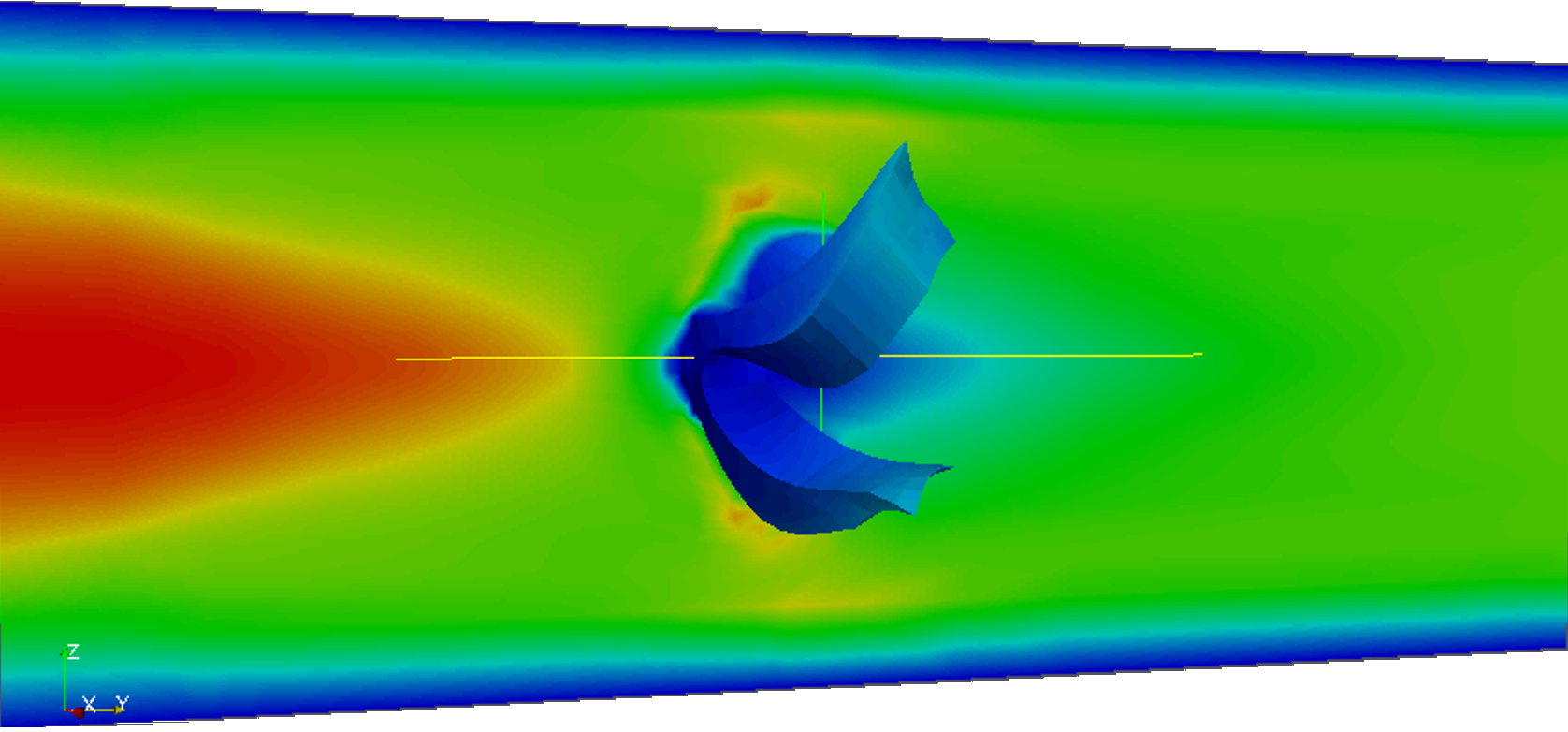}
      \end{minipage}}}
\end{tabular}
} \caption{A more flexible and thinner elastic rotor which deforms
while rotates: the evolution of magnitude of velocity field with
time on the central cross section along flow direction.
\label{fig:softsolid}}
\end{figure}

Now we conduct a consistency check for the developed structure
model. We carry out a series of numerical computations for the
fluid-turbine interaction with respect to an increasing Young's
modulus of the turbine, for instance, from $E=2.5\times 10^4$ Pa all
the way up to $E=2.5\times 10^{9}$ Pa, which means the turbine tends
to be stiffer and stiffer toward a rigid body. Because a rigid body
does not deform, one shall expect that the deformation displacement
in any place of the turbine approaches to zero along with an
increasing Young's modulus. To that end, we just need to check the
displacement component in flow direction since there is not any the
component of rotational displacement existing in that direction. By
the configuration, we set the turbine to spin only in the plane
which is perpendicular to the flow direction. In addition, since the
blade tip is the thinnest part in the turbine, it bears relatively
the largest deformation during the interaction with the fluid. Thus,
by checking the variation tendency of the blade tip displacement in
flow direction along with the increasing Young's modulus, we are
able to validate our developed structure model in the sense of an
asymptotic stiffness. Fig. \ref{fig:BladetipDisplacement} shows a
semi logarithmic plot of the variation of blade tip displacement in
flow direction along with time and the increasing Young's modulus,
where the displacement component is logarithmically scaled in order
to distinctly illustrate that the blade tip displacement approaches
to zero while the Young's modulus becomes much larger at any
instantaneous time, demonstrating that the developed structure model
is consistently correct and the developed numerical method is
convergent with respect to a crucial physical parameter of structure
that may vary, asymptotically.
\begin{figure}[htbp]
\begin{center}
\graphicspath{{:figure:}}
\includegraphics[width=0.65\textwidth]{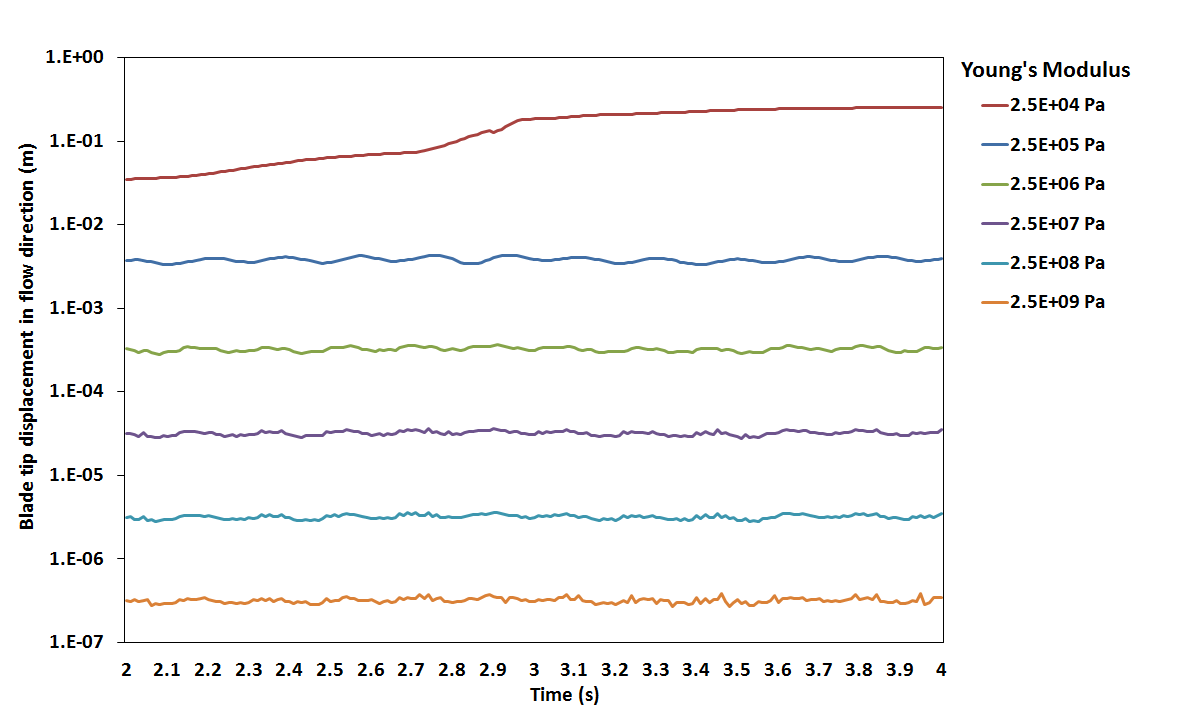}
\end{center}
\caption{The variation tendency of blade tip displacement in flow
direction at the central line of the rotor thickness along with the
increasing Young's modulus and time, where the ordinate is scaled
logarithmically.} \label{fig:BladetipDisplacement}
\end{figure}


Next, we apply our validated monolithic ALE method to a practical
fluid-structure interaction problem involving a realistic
three-dimensional hydro-turbine which is still under construction
for a new hydro plant. As illustrated in Figs.
\ref{fig:3dturbine}-\ref{fig:realturbine}, a realistic turbine bears
five blades curved in three-dimensional fashion and is immersed in
the fluid domain with a shape of circular truncated cone. Thus, we
take advantage of the shape of fluid flow channel and make the
entire fluid domain as rotational, thus no stationary fluid domain
exists in this example. By doing that way we simply solve the ALE
mapping equation in the entire fluid domain to obtain the entire
fluid mesh with zero boundary condition on the outer boundary of
$\hat\Omega_f$. Under the same prescribed incoming velocity and
angular velocity as given for the previous self-defined rotor, we
obtain the similar illustrations for the numerical results shown in
Figs.
\ref{fig:realvelocitymagnitude}-\ref{fig:realcross3dstreamline}.
Fig. \ref{fig:realvelocitymagnitude} displays the development of
velocity magnitude with time marching in a quarter part of fluid
domain along the flow direction, and Fig.
\ref{fig:realcross3dstreamline} shows the velocity field in the
streamline as well as the vector version.
\begin{figure}[!htb]
\begin{center}
\includegraphics*[height=1.in,width=1.in]{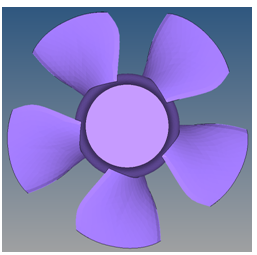}
\hspace*{0.1in}
\includegraphics*[height=1.in,width=.8in]{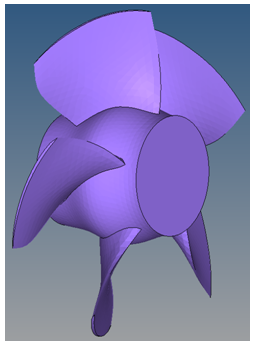}
\hspace*{0.1in}
\includegraphics*[height=1.in,width=2.in]{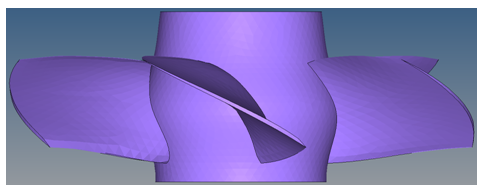}
\end{center}
\caption{A realistic 3D turbine. \label{fig:3dturbine}}
\end{figure}
\begin{figure}[htbp]
\begin{center}
\graphicspath{{:figure:}}
\includegraphics*[height=1.2in,width=2.in]{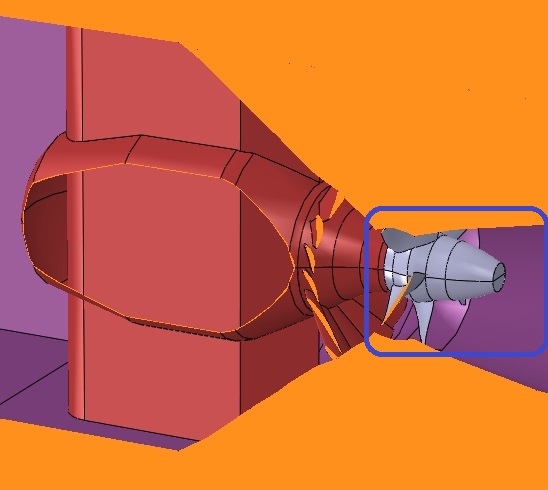}
\hspace*{0.0in}
\includegraphics*[height=1.2in,width=2.in]{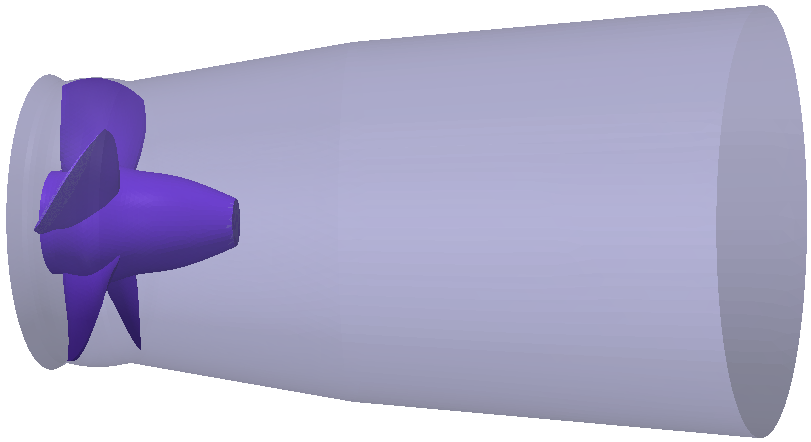}
\hspace*{0.0in}
\includegraphics*[height=1.2in,width=2.in]{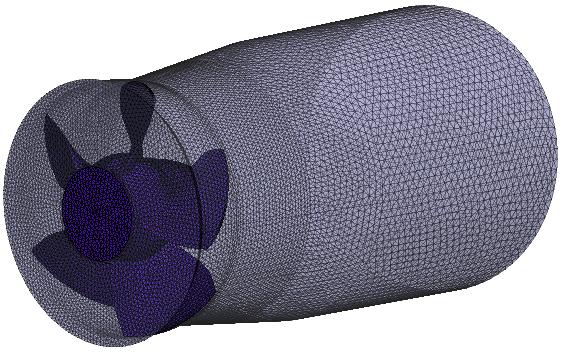}
\end{center}
\caption{The realistic turbine immersed in the fluid domain with the
shape of a circular truncated cone: $\Omega=\Omega_s+\Omega_f$ and
$\Omega_f=\Omega_{rf}$, no stationary fluid domain $\Omega_{sf}$,
and the only interface is
$\Gamma=\partial\Omega_{f}\cap\partial\Omega_{s}$.}
\label{fig:realturbine}
\end{figure}

\begin{figure}[htdp]
\centerline{
\begin{tabular}{c}
\parbox{1.9in}{ \subfigure[Time=1.6s]{
   \begin{minipage}[t]{1\linewidth}
      \centering
      \includegraphics*[height=.9in,width=1.9in]{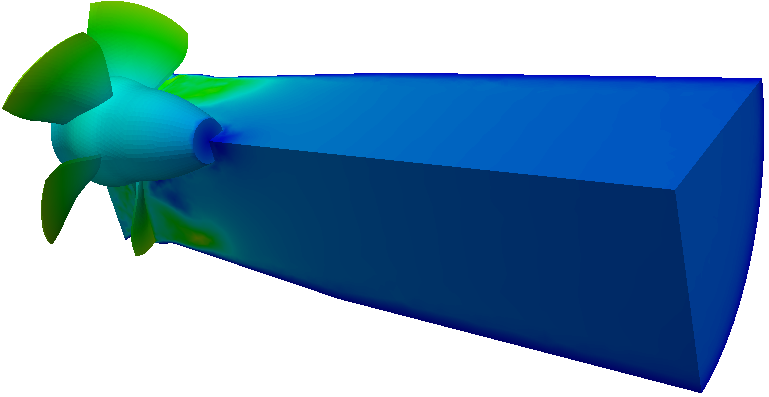}
      \end{minipage}}}
\hspace*{0.0in}
\parbox{1.9in}{\subfigure[Time=3.15s]
{
   \begin{minipage}[t]{1\linewidth}
      \centering
      \includegraphics*[height=.9in,width=1.9in]{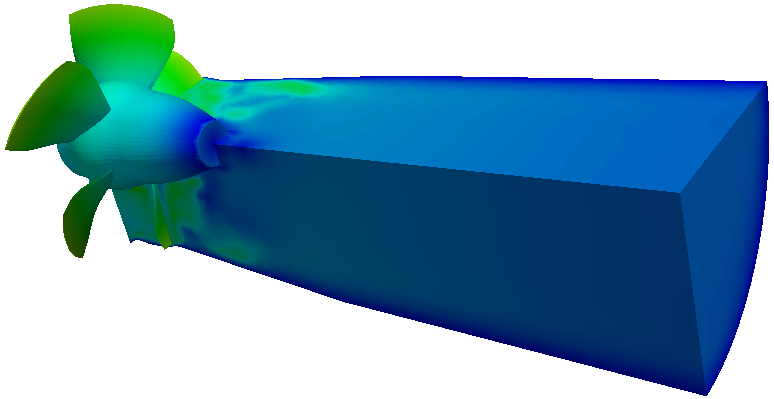}
      \end{minipage}}}
\hspace*{0.0in}
\parbox{1.9in}{\subfigure[Time=4.75s]
{
   \begin{minipage}[t]{1\linewidth}
      \centering
      \includegraphics*[height=.9in,width=1.9in]{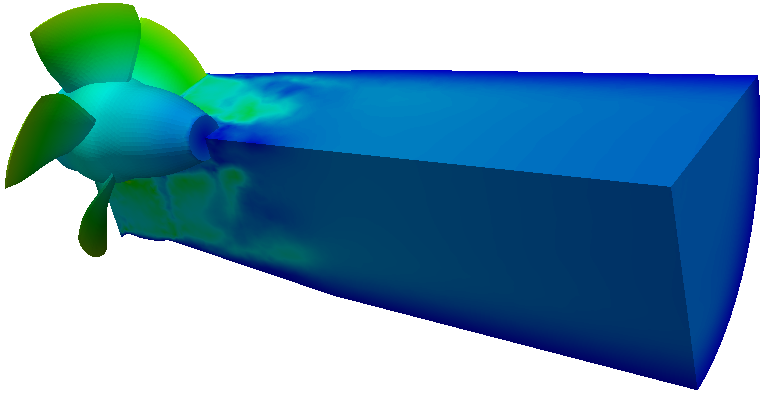}
      \end{minipage}}}
\hspace*{0.0in}
\parbox{.4in}{ 
   \begin{minipage}[t]{1\linewidth}
      \centering
      \includegraphics*[height=1.in,width=.4in]{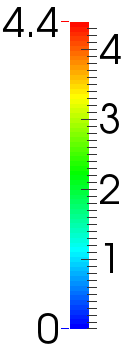}
      \end{minipage}}
\\
\parbox{1.9in}{ \subfigure[Time=6.35s]{
   \begin{minipage}[t]{1\linewidth}
      \centering
      \includegraphics*[height=.9in,width=1.9in]{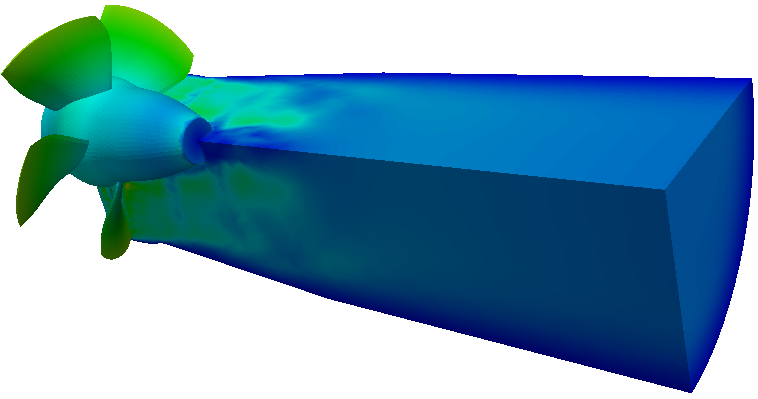}
      \end{minipage}}}
\hspace*{0.0in}
\parbox{1.9in}{\subfigure[Time=7.9s]
{
   \begin{minipage}[t]{1\linewidth}
      \centering
      \includegraphics*[height=.9in,width=1.9in]{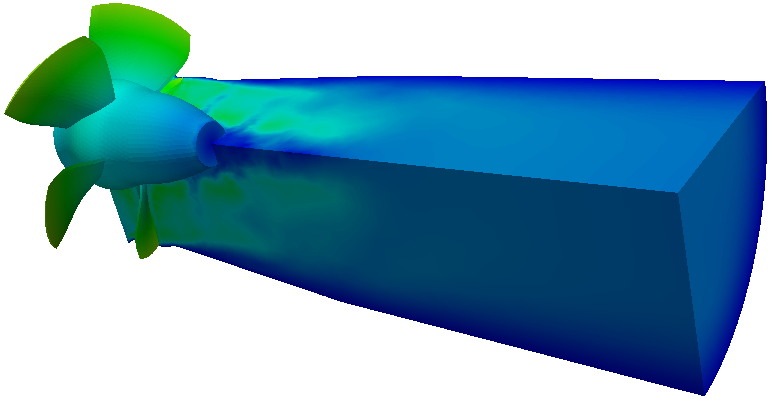}
      \end{minipage}}}
\hspace*{0.0in}
\parbox{1.9in}{\subfigure[Time=9.5s]
{
   \begin{minipage}[t]{1\linewidth}
      \centering
      \includegraphics*[height=.9in,width=1.9in]{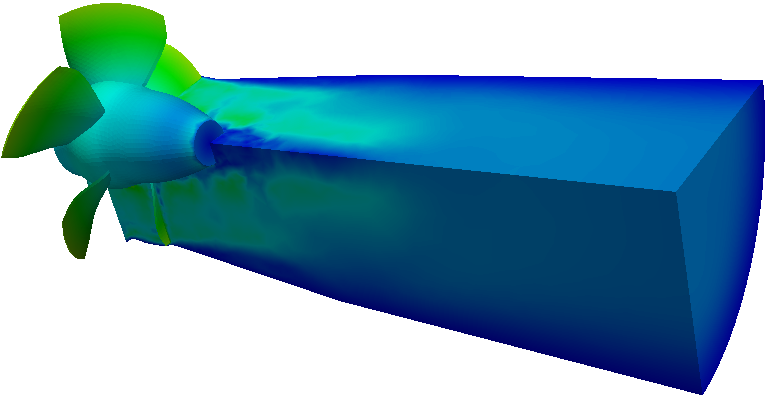}
      \end{minipage}}}
\hspace*{0.0in}
\parbox{.4in}{ 
   \begin{minipage}[t]{1\linewidth}
      \centering
      \includegraphics*[height=1.in,width=.4in]{vrealscale.png}
      \end{minipage}}
\\
\parbox{1.9in}{ \subfigure[Time=11.05s]{
   \begin{minipage}[t]{1\linewidth}
      \centering
      \includegraphics*[height=.9in,width=1.9in]{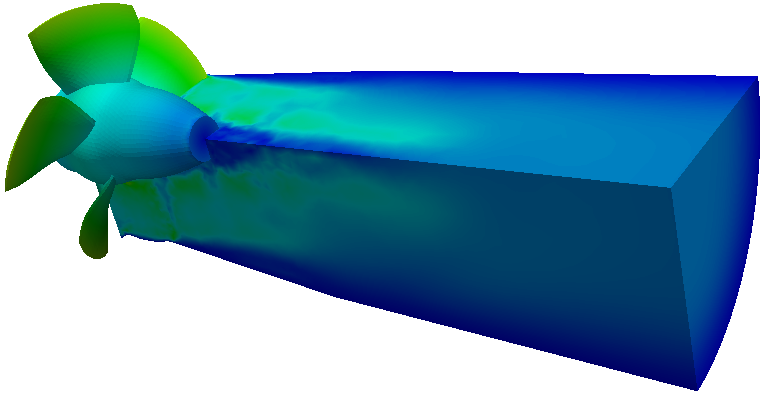}
      \end{minipage}}}
\hspace*{0.0in}
\parbox{1.9in}{\subfigure[Time=12.65s]
{
   \begin{minipage}[t]{1\linewidth}
      \centering
      \includegraphics*[height=.9in,width=1.9in]{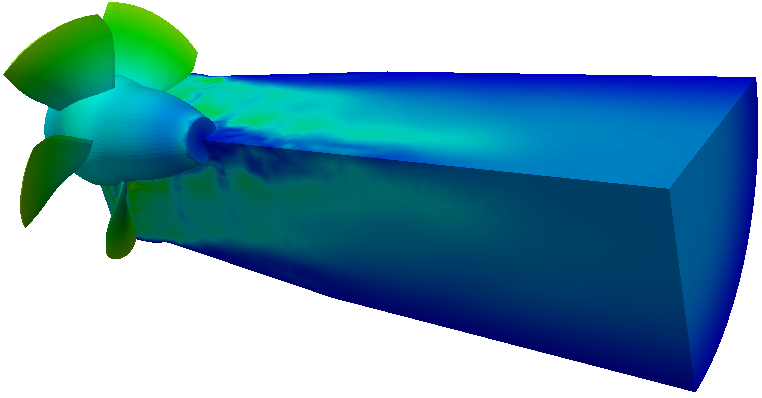}
      \end{minipage}}}
\hspace*{0.0in}
\parbox{1.9in}{\subfigure[Time=14.25s]
{
   \begin{minipage}[t]{1\linewidth}
      \centering
      \includegraphics*[height=.9in,width=1.9in]{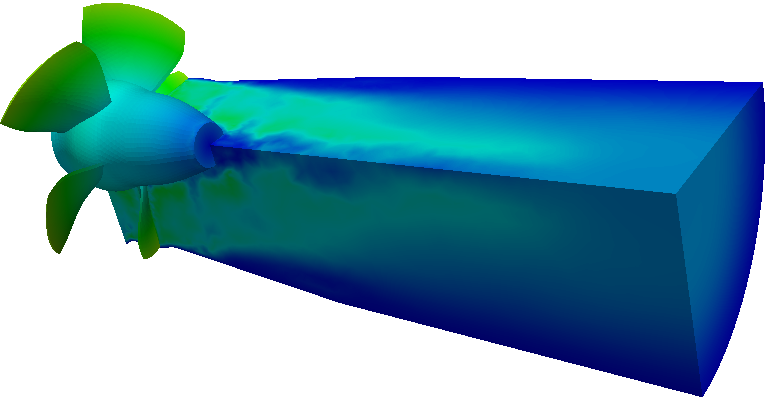}
      \end{minipage}}}
\hspace*{0.0in}
\parbox{.4in}{ 
   \begin{minipage}[t]{1\linewidth}
      \centering
      \includegraphics*[height=1.in,width=.4in]{vrealscale.png}
      \end{minipage}}
\end{tabular}
} \caption{A realistic elastic rotor: the evolution of magnitude of
velocity with time in 1/4 fluid domain.}
\label{fig:realvelocitymagnitude}
\end{figure}

\begin{figure}[htdp]
\centerline{
\parbox{1.8in}{ \subfigure[Magnitude of velocity field]{
   \begin{minipage}[t]{1\linewidth}
      \centering
\includegraphics[height=1.1in,width=1.8in]{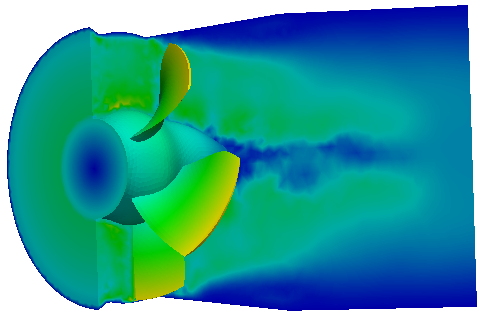}
      \end{minipage}}}
\hspace*{0.0in}
\parbox{2.in}{ \subfigure[Streamline field]{
   \begin{minipage}[t]{1\linewidth}
      \centering
\includegraphics[height=1.in,width=2.in]{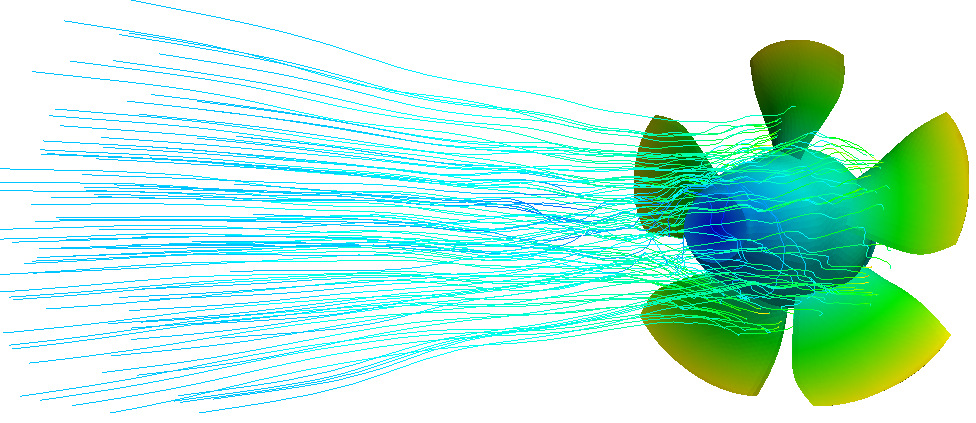}
      \end{minipage}}}
\hspace*{0.0in}
\parbox{2.in}{\subfigure[Velocity field]
{
   \begin{minipage}[t]{1\linewidth}
      \centering
\includegraphics*[height=1.2in,width=2.in]{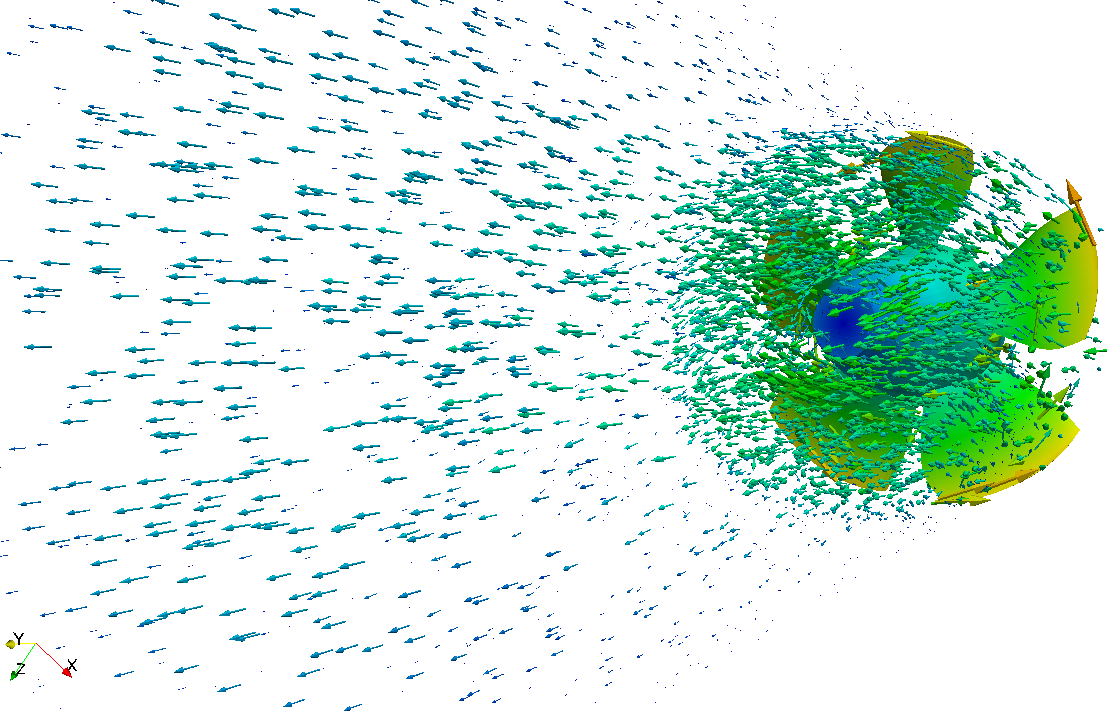}
      \end{minipage}}}
\hspace*{0.0in}
\parbox{.4in}{ 
   \begin{minipage}[t]{1\linewidth}
      \centering
      \includegraphics*[height=1.in,width=.4in]{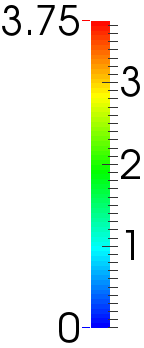}
      \end{minipage}}
} \caption{A realistic elastic rotor: magnitude of velocity field on
the central cross section along flow direction (left); streamline
(middle) and velocity (right) fields at 14.25s.
\label{fig:realcross3dstreamline}}
\end{figure}

\begin{remark} In some senses, our current model and numerical method also work for the passive
rotational case, i.e., for which there is not a prescribed angular
velocity on the axis of rotation. In this case, we do not know the
structure velocity on the cylindrical axis of rotation
$\Gamma_{in}$. In order to make the initially static turbine
spinning up while the fluid flow starts to impact, we can shrink the
cylindrical axis of rotation as a point axis (or a line axis in 3D)
on its barycenter, and fix the structure displacement and then
velocity as zero on the point/line axis, thus the turbine is not be
carried away by the fluid flow but likely rotates about the
point/line axis if the external fluid force is acted on the turbine.
This configuration is inaccurate in the mathematical point of view
but numerically works to some extent, with which we are still able
to apply our developed model and numerical methods to the FSI
problem involving with a passively rotational hydro-turbine, and
observe that the elastic turbine starts to spin up from its
initially static position, and accelerates its rotation until
reaching a steady rotational status, i.e., the fluid impact force
acting on the turbine attains a balance with the resistance force
arising from the fluid itself due to the viscosity, if there is not
any other external torque being exerted on the turbine.

However, we do not intend to illustrate the numerical results of
passive rotational case in this paper since the inaccurate
configuration of (d-2)-dimensional axis does not make fully
mathematical sense, and additionally, might also introduce extra
stress concentration effect on such (d-2)-dimensional axis under a
large external load. Recently we develop a novel and more accurate
numerical method for the passive rotational case, which will be the
subject of a forthcoming paper.
\end{remark}

\section{Conclusions}
\label{sec:conclusion} We build an Eulerian-Lagrangian model for
fluid-structure interaction problem involving with an elastic rotor
based on the arbitrary Lagrangian Eulerian (ALE) approach. Using
velocity as the principle unknown in structure equation, we are able
to deal with the no-slip condition on the interface of fluid and
structure more flexibly with the technique of Master-Slave
Relations. In addition, with the variational formulation and mixed
finite element method, the interface conditions are automatically
enforced. Our proposed novel ALE method can flexibly generate a
rotational and deformable fluid mesh according to the boundary
conditions arising from the ambient elastic rotor and the stationary
fluid. The linear system resulting from linearization and discretization is proved to be well-posed. By means of the developed monolithic algorithm involving the
relaxed fixed-point iteration, a series of satisfactory numerical
results are illustrated and validated for the elastic hydro-turbine
that is actively spinning around its axis of rotation and
interacting with the fluid, demonstrating that our model and
numerical techniques are effective to explore the interactional
mechanism between the fluid and an elastic rotor. As a part of the
future work, we will continue to develop the numerical method for
the case of passive rotation that is more suitable for the
hydro-turbine, and further, study the FSI problem in which the
turbulence flow is involved.

\section*{Acknowledgments}
J. Xu, L. Wang, and K. Yang were partially supported by the U.S.
Department of Energy, Office of Science, Office of Advanced
Scientific Computing Research as part of the Collaboratory on
Mathematics for Mesoscopic Modeling of Materials under contract
number {DE-SC0009249}, also were partially supported by National
Natural Science Foundation of China (NSFC) (Grant No. 91430215). P.
Sun was partially supported by NSF Grant DMS-1418806 and UNLV
Faculty Opportunity Awards, also by {DE-SC0009249} during his
sabbatical leave at Pennsylvania State University in 2013-2014. L.
Zhang was partially supported by NSFC (Grant No. 51279071) and
Doctoral Foundation of Ministry of Education of China (Grant No.
20135314130002). We also appreciate the valuable helps from Dr.
Xiaozhe Hu and FASP group of J. Xu for an efficient
linear algebraic solver.


\end{document}